\documentclass[a4paper,9pt]{article}
\usepackage{amssymb,amsmath,amsthm,graphicx,fancyhdr}

\usepackage{color}

\newtheorem{Thm}{Theorem}

\newtheorem{Prop}{Proposition}
\newtheorem{Lem}{Lemma}
\newtheorem{Rem}{Remark}
\newtheorem{Cor}{Corollary}

\newcommand{\C}{\mathbb{C}}
\newcommand{\R}{\mathbb{R}}
\newcommand{\Z}{\mathbb{Z}}
\newcommand{\N}{\mathbb{N}}
\newcommand{\T}{\mathbb{T}}

\renewcommand{\Im}{\operatorname{Im}}
\renewcommand{\Re}{\operatorname{Re}}
\newcommand{\I}{\infty}

\newcommand{\norm}[1]{\left\lVert #1\right\rVert}
\newcommand{\Lebn}[2]{\left\lVert #1 \right\rVert_{L^{#2}}}
\newcommand{\Sobn}[2]{\left\lVert #1 \right\rVert_{H^{#2}}}
\newcommand{\hSobn}[2]{\left\lVert #1 \right\rVert_{\dot{H}^{#2}}}

\newcommand{\tnorm}[1]{\lVert #1\rVert}

\def\({\left(}
\def\){\right)}
\def\<{\left\langle}
\def\>{\right\rangle}

\def\d{{\partial}}

\newcommand{\al}{\alpha}

\newcommand{\ep}{\varepsilon}

\newcommand{\la}{\lambda}

\newcommand{\de}{\delta}
\newcommand{\De}{\Delta}

\title
{Construction of blow-up solutions for Zakharov system on $\T^2$}
\author{Nobu Kishimoto${}^{\dagger}$ and Masaya Maeda${}^{\ddagger}$}

\date{}

\begin{document}

\maketitle

\vskip-5mm
\centerline{${}^{\dagger}$Department of Mathematics, Kyoto University,}
\centerline{Kyoto, 606-8502, Japan}

\vskip3mm

\centerline{${}^{\ddagger}$Mathematical institute, Tohoku University,}
\centerline{Sendai, 980-8578, Japan}

\begin{abstract}
We consider the Zakharov system in two space dimension with periodic boundary condition:
\begin{equation}\tag{$\mathrm{Z}$}
\left\{
\begin{aligned}
&i \d_t u = -\De u + nu,\\
& \d_{tt}n= \De n + \De |u|^2,\qquad (t,x)\in [0,T) \times \T^2 .
\end{aligned}
\right.
\end{equation}
We prove the existence of finite time blow-up solutions of (Z).
Further, we show there exists no minimal mass blow-up solution.
\end{abstract}


\section{Introduction}
In this paper, we consider the Zakharov system on $\T^2=(\R/2\pi\Z)^2$:
\begin{equation}\label{eq:Z}\tag{$\mathrm{Z}$}
\left\{
\begin{aligned}
&i \d_t u = -\De u + nu,\\
& \frac{1}{c_0^2}\d_{tt}n= \De n + \De |u|^2,\\
&u(0)=u_0,\ n(0)=n_0,\ n_t(0)=n_1,
\end{aligned}
\right.
\end{equation}
where $c_0>0$, $u:[0,T)\times\T^2\rightarrow \C$, $n:[0,T)\times\T^2\rightarrow\R$ and $u_0$, $n_0$, $n_1$ are initial data.
Further, in our results, we fix $c_0=1$.

Zakharov system was introduced in \cite{Z72} to describe the collapse of Langmuir wave (or electron plasma waves) in a non-magnetized plasma.
In the context of the dynamics of Langmuir wave, $u$ represents the slowly varying envelope of the electric field and $n$ denotes the deviation of the ion density from its mean value.
From the physical point of view, the evolution of (\ref{eq:Z}) leads to the formation of a cavity of ion density and an amplification of the amplitude of the electric field.
Further, the collapse of cavity gives an explanation for the mechanism of the dissipation of long-wavelength plasma waves.
Therefore, the wave collapse, which is the finite time blowup of the solution of (\ref{eq:Z}), plays a central role in the strong turbulence of Langmuir waves.
See, for example \cite{Robinson97, SulemSulem}.

Note that the subsonic limit of Zakharov system ($c_0\rightarrow \infty$) formally gives us the nonlinear Schr\"odinger equation with critical exponent:
\begin{equation}\label{eq:NLS}\tag{$\mathrm{NLS}$}
\left\{
\begin{aligned}
&i \d_t u = -\De u - |u|^2u,\\
&u(0)=u_0.
\end{aligned}
\right.
\end{equation}
We say critical because it is the smallest power that admits blowup in finite time for initial data in the energy class ($u_0\in H^1$).

The local-in-time well-posedness for these equations have been extensively studied.
For a moment we use $X$ to denote $\R^2$ or $\T^2$.
For \eqref{eq:NLS}, it is known for initial data in $H^s(X)$, with $s\geq 0$ for the case $\R^2$ (\cite{CW90}) and with $s>0$ for the case $\T^2$ (\cite{B93-1}).
See \cite{TakaTzv01} for a result on $\R \times \T$.
Moreover, for the strong solution $u\in C([0,T];H^s(X))$ obtained, we have the conservation of mass
\[ \Lebn{u(t)}{2}=\Lebn{u_0}{2},\quad \forall t\in [0,T]\]
and, if $s\geq 1$, the conservation of energy
\[ \frac{1}{2}\Lebn{\nabla u(t)}{2}^2-\frac{1}{4}\Lebn{u(t)}{4}^4=\frac{1}{2}\Lebn{\nabla u_0}{2}^2-\frac{1}{4}\Lebn{u_0}{4}^4,\quad \forall t\in [0,T] .\]

The Zakharov system \eqref{eq:Z} has similar conservation laws.
First, the mass of $u(t)$ is also conserved.
In addition, assume $n_1\in \hat{H}^{-1}(X;\R)$, where
\[ \hat{H}^{-1}(X;\R):=\left\{ \phi \in H^{-1}(X;\R )|\;\exists \psi \in L^2(X;\R ^2)~\text{s.t.}~\phi =-\nabla \psi\right\}\]
(see also Remark~\ref{hatH} below).
Then, $\d_t n(t)\in \hat{H}^{-1}$ for all $t$ and the wave part of \eqref{eq:Z} may be written in the form
\[ \d_t n+\nabla v=0;\qquad \frac{1}{c_0^2}\d_t v+\nabla n=-\nabla (|u|^2),\]
for some $v(t)\in L^2(X;\R ^2)$.
In this case, we have the conservation of energy $\mathcal{E}(t)=\mathcal{E}(0)$, where $\mathcal{E}$ is defined by
\begin{eqnarray}\label{eq:energy}
\mathcal{E}=\mathcal{E}(u,n,v):=\Lebn{\nabla u}{2}^2+\frac{1}{2}\( \Lebn{n}{2}^2+\Lebn{v}{2}^2\) +\int _Xn|u|^2\,dx.
\end{eqnarray}
The local well-posedness of \eqref{eq:Z} on $\R^2$ in the energy space $H^1\times L^2\times \hat{H}^{-1}$ was first obtained by Bourgain and Colliander \cite{BC96}, which was improved to $H^1\times L^2\times H^{-1}$ and wider spaces by Ginibre, Tsutsumi, and Velo \cite{GTV97}.
The lowest regularity in which the local well-posedness is known so far is $L^2\times H^{-1/2}\times H^{-3/2}$ (\cite{BHHT09}). 
The case $\T^2$ is more involved, but the local well-posedness in the energy space (actually in $H^1\times L^2\times H^{-1}$ and some wider spaces) was recently proved, see \cite{Ki}.

\begin{Rem}\label{hatH}
The definition of the $\hat{H}^{-1}$ norm of $\phi=-\nabla \vec{\psi}$ is a bit tricky.
It would not be well-defined if we used simply the $L^2$ norm of $\vec{\psi}$.
For example, consider $\vec{\hat{\psi}}=(\chi _B,-\frac{\xi _2}{\xi _1}\chi _B)$ for $X=\R^2$, where $B:=[1,2]^2\cup [-2,-1]^2$ and $\chi _B$ is the characteristic function of $B$, and $\vec{\psi}=(\cos (x_1+x_2),-\cos (x_1+x_2))$ for $X=\T^2$.
Then, $||\vec{\psi}||_{L^2}>0$ but $\phi =-\nabla \vec{\psi}=0$.

Before the definition we recall the Helmholtz decomposition $L^2(X;\R)^2=L^2_\sigma \oplus ^\perp G$ into the solenoidal space $L^2_\sigma =\{ \vec{\psi}:\nabla \vec{\psi}=0\}$ and the gradient space $G=\{ \nabla \eta:\eta \in \dot{H}^1(X;\R)\}$.
Let $\mathbb{P}_G:L^2(X;\R)^2\to G$ be the orthogonal projection onto $G$.
We now define $||\phi ||_{\hat{H}^{-1}}:=||\mathbb{P}_G\vec{\psi}||_{L^2}$ for $\phi =-\nabla \vec{\psi}$, which is a well-defined norm on $\hat{H}^{-1}$.


\end{Rem}

The time global existence and blow-up problem of (\ref{eq:NLS}) on $\R^2$ and $\T^2$ have also been studied by many authors \cite{G77, MR03, MR04, MR05-2, MR05, MR05-3, MR06, MT90, Nawa99, OT89, OT91, OT91-2, Weinstein82}.
It is well known that if $u_0\in H^1(\R^2)$ and $||u_0||_{L^2(\R^2)}<||Q||_{L^2(\R^2)}$, then the solution of (\ref{eq:NLS}) on $\R^2$ exists globally in time.
(In fact, it was recently shown by Dodson \cite{D} that the assumption $u_0\in H^1$ can be replaced with $u_0\in L^2$.) 
Here, $Q$ is the unique positive radial solution of
\begin{equation}\label{eq:Q}
-\De Q + Q - Q^3 = 0,\ x\in\R^2.
\end{equation}
This also holds for the case $\T^2$.
That is, if $u_0\in H^1(\T^2)$ and $||u_0||_{L^2(\T^2)}<||Q||_{L^2(\R^2)}$, then the solution of (\ref{eq:NLS}) on $\T^2$ exists globally in time.
On the other hand, if $M\geq ||Q||_{L^2(\R^2)}^2$, it is known that there exists $u_0\in H^1(X)$ such that $||u_0||_{L^2(X)}^2=M$ and the solution of (\ref{eq:NLS}) on $X$ blows up in finite time (for the case $\T^2$ see \cite{Antonini02}).
In this sense, $||Q||_{L^2(\R^2)}$ is the sharp threshold for the global existence and blowup of (\ref{eq:NLS}) both on $\R^2$ and on $\T^2$.
For the blow-up problem of (\ref{eq:Z}) on $\R^2$, Glangetas and Merle \cite{GM94} constructed a blow-up solution with $||u_0||_{L^2(\R^2)}$ arbitrarily near $||Q||_{L^2}$.
Further, in \cite{GM94b} they showed that if $||u_0||_{L^2(\R^2)}\leq ||Q||_{L^2(\R^2)}$, then the solution of (\ref{eq:Z}) on $\R^2$ with $(u(0),n(0),n_t(0))=(u_0,n_0,n_1)\in H^1(\R^2)\times L^2(\R^2)\times H^{-1}(\R^2)$ exists globally in time.
However, it seems there is no counterpart of the results of Glangetas and Merle for $\T^2$  as far as the authors know.

In this paper, we construct a blow-up solution of (\ref{eq:Z}) by using the fixed point argument.
Further, as in the $\R^2$ case, we show that if $||u_0||_{L^2(\T^2)}\leq ||Q||_{L^2(\R^2)}$, then the solution of \eqref{eq:Z} exists globally in time.
For (\ref{eq:NLS}), Burq, G\'erard, and Tzvetkov \cite{BGT03} constructed a blow-up solution on $\T^2$ by adapting an idea of Ogawa and Tsutsumi \cite{OT89}, who treated a similar problem on $\T$.
In \cite{BGT03} they cut off the explicit blow-up solution on $\R^2$ and solved the perturbed equation.
Thus we use the blow-up solution of (\ref{eq:Z}) on $\R^2$ which was constructed by Glangetas and Merle.
However, in contrast to (\ref{eq:NLS}), (\ref{eq:Z}) has a derivative in the nonlinearity.
Therefore, we cannot directly apply the argument of \cite{BGT03} because of the so-called ``loss of derivative.''
To overcome this difficulty, we introduce the modified energy and derive an a priori estimate for the approximate solutions.
Our main result is as follows.
\begin{Thm}\label{Thm:blowup}
For arbitrary $M>||Q||_{L^2(\R^2)}^2$, there exists $T=T(M)>0$ and a solution $(u,n)$ of \eqref{eq:Z} in the class $(u,n,n_t)\in C([0,T); H^1(\T^2)\times L^2(\T^2)\times \hat{H}^{-1}(\T^2))$ with the following properties.
\begin{enumerate}
\item[$(\mathrm{i})$]
$||u(t)||_{L^2(\T^2)}^2<M$.
\item[$(\mathrm{ii})$]
$C_1(T-t)^{-1}\leq ||u(t)||_{H^1(\T^2)}+||n(t)||_{L^2(\T^2)}+||n_t(t)||_{\hat{H}^{-1}(\T^2)} \leq C_2(T-t)^{-1}$ for some $C_1,C_2>0$.
\item[$(\mathrm{iii})$]
$\lim\limits _{t\to T}||\nabla u(t)||_{L^2(\T^2\setminus B(0,r))}=0$, $\lim\limits _{t\to T}||n(t)||_{L^2(\T^2\setminus B(0,r))}=0$ for any $r>0$ sufficiently small, where $B(a,r):=\{x\in\T^2 |\ |x-a|<r\}$.

\end{enumerate}
\end{Thm}
Our approach is easily generalized to the case of exactly $p$ blow-up points.
A similar generalization was mentioned by Godet~\cite{Godet} for \eqref{eq:NLS}.
\begin{Cor}\label{cor:blowup}
Let $\{ x_1,\dots ,x_p\}$ be distinct points in $\T^2$.
For arbitrary $M>p||Q||_{L^2(\R^2)}^2$, there exists $T=T(M)>0$ and a solution $(u,n)$ of \eqref{eq:Z} in the class $(u,n,n_t)\in C([0,T); H^1(\T^2)\times L^2(\T^2)\times \hat{H}^{-1}(\T^2))$ with the following properties.
\begin{enumerate}
\item[$(\mathrm{i})$]
$||u(t)||_{L^2(\T^2)}^2<M$.
\item[$(\mathrm{ii})$]
$C_1(T-t)^{-1}\leq ||u(t)||_{H^1(\T^2)}+||n(t)||_{L^2(\T^2)}+||n_t(t)||_{\hat{H}^{-1}(\T^2)} \leq C_2(T-t)^{-1}$ for some $C_1,C_2>0$.
\item[$(\mathrm{iii})$]
$\lim\limits _{t\to T}(T-t)||\nabla u(t)||_{L^2(B(x_j,r))}>0$, $\lim\limits _{t\to T}(T-t)||n(t)||_{L^2(B(x_j,r))}>0$, and $\lim\limits _{t\to T}||\nabla u(t)||_{L^2(\T^2\setminus \cup _{j=1}^pB(x_j,r))}=0$, $\lim\limits _{t\to T}||n(t)||_{L^2(\T^2\setminus \cup _{j=1}^pB(x_j,r))}=0$ for any $r>0$ sufficiently small and any $1\leq j\leq p$.
\end{enumerate}
\end{Cor}

The global existence of the solution for the case $||u_0||_{L^2(\T^2)}\leq ||Q||_{L^2(\R^2)}$ is a corollary of the following mass concentration result.

\begin{Thm}\label{thm:concentration}
Suppose $(u,n)$ is the solution of \eqref{eq:Z} which blows up at $t=T\in (0,\infty )$.
Then, there exists $t_n\rightarrow T$ and $y_n\in\T^2$ such that
\begin{eqnarray*}
\liminf_{n\rightarrow \infty}\int_{|x-y_n|<R}|u(t_n,x)|^2\,dx\geq ||Q||_{L^2(\R^2)}^2
\end{eqnarray*}
for any $R>0$.
\end{Thm}

\begin{Cor}\label{cor:nonexistence}
If $(u_0,n_0,n_1)\in H^1(\T^2)\times L^2(\T^2)\times H^{-1}(\T^2)$ satisfies $||u_0||_{L^2(\T^2)}\leq ||Q||_{L^2(\R^2)}$, then the corresponding solution of \eqref{eq:Z} exists globally in time.
\end{Cor}

Corollary \ref{cor:nonexistence} can be derived by using the same argument as Glangetas and Merle with a sharp Gagliardo-Nirenberg inequality on $\T^2$ by Ceccon and \mbox{Montenegro} \cite{CM08}.
However, for the proof of Theorem \ref{thm:concentration}, it is not sufficient by itself to replace the sharp Gagliardo-Nirenberg inequality on $\R^2$ with that on $\T^2$.
This is because the terms $||\nabla u||_{L^2}^2$ and $||u||_{L^2}^2$ appearing in the Gagliardo-Nirenberg inequality on $\T^2$ have different scalings.
Therefore, we use a concentration compactness argument and split $u$ in many pieces so that we can use the Gagliardo-Nirenberg inequality on $\R^2$.

This paper is organized as follows.
In section \ref{formulation}, we formulate the perturbed equation which we have to solve to construct a finite time blow-up solution.
In section \ref{regularization}, we construct an approximate solution by regularizing the perturbed equation.
In section \ref{modified}, we introduce a modified energy and derive an a priori estimate for the approximate solution.
This estimate will allow us to construct a solution to the perturbed equation. 
Also, the idea for multi-point blowup is given in section \ref{modified}.
In section \ref{sec:concentration}, we prove Theorem \ref{thm:concentration} and Corollary \ref{cor:nonexistence}.
In the appendix of this paper, for the readers convenience, we give a brief sketch of the proof of the modified concentration compactness lemma which we will use for the proof of Theorem \ref{thm:concentration}.

We define some notations which we use in the following.
We denote the Fourier series of $u(t,x)$ in the spatial variable as
\begin{eqnarray*}
u(t,x)=\sum_{m\in\Z^2}e^{imx}\hat{u}(t,m).
\end{eqnarray*}
We define the Sobolev spaces $H^k(\T^2)$ for $k\in \R$ as
\begin{eqnarray*}
H^k(\T^2)&:=&\{u\in \mathcal{D}'(\T^2)\ |\ ||u||_{H^k}<\infty\},\\
||u||_{H^k}^2&:=&\sum_{m\in\Z^2}\<m\>^{2k}|\hat{u}(m)|^2,
\end{eqnarray*}
where $\<x\>:=(1+|x|^2)^{1/2}$.
We write $A\lesssim B$ to denote the estimate $A\leq CB$ with a constant $C>0$, which may depend on some parameters in a harmless way, and write $A\sim B$ if $A\lesssim B\lesssim A$.
We use the notations like $\lesssim _{\ep ,\la}$ when we need to emphasize the dependence of constants on some parameters.


\section{Formulation}\label{formulation}

First of all, we recall the result by Burq, G\'erard, and Tzvetkov~\cite{BGT03}, which constructed blow-up solutions to \eqref{eq:NLS} on $\T^2$.
\eqref{eq:NLS} on $\R^2$ has a family of explicit blow-up solutions $\{ \tilde{R}_\la \} _{\la >0}$ which blow up at $t=T$, where
\[ \tilde{R}_{\la}(t,x)=\frac{1}{\la(T-t)}e^{i\(\frac{1}{\la^2(T-t)}-\frac{|x|^2}{4(T-t)}\)}Q\(\frac{x}{\la(T-t)}\) ,\]
and $Q$ is given in (\ref{eq:Q}).
Let $\psi\in C_{0,r}^{\infty}(\R^2)$ be such that $0\leq \psi \leq 1$, supp\,$\psi\subset \{|x|<2\}$ and $\psi(x)=1$ for $|x|<1$, then the function $\psi \tilde{R}_\la (t)$, which is restricted on a ball $B(0,2)\subset [-\pi ,\pi ]^2$, can be regarded as a function on $\T^2$.
Consider the function
\[ u(t,x)=\psi (x) \tilde{R}_\la (t,x)+v(t,x)\]
with $v:[0,T]\times \T^2\to \C$.
Then, $u$ is a blow-up solution to \eqref{eq:NLS} on $\T^2$ if $v$ solves the equation
\begin{equation}\label{eq:BGTs}
\left\{
\begin{aligned}
(i \d_t+\De) v = &~Q_{2,3}(v)-\psi ^2(\tilde{R}_\la ^2\bar{v}+|\tilde{R}_\la |^2v)\\
&+(1-\psi ^2)\psi |\tilde{R}_{\la}|^2\tilde{R}_\la -2\nabla\psi\nabla\tilde{R}_{\la}-\De\psi\tilde{R}_{\la},\\
v(t)\to 0\quad &\text{as }t\to T,
\end{aligned}
\right. 
\end{equation}
where quadratic and cubic terms with respect to $v$ have been written as $Q_{2,3}(v)$.

Applying the fixed point argument to the associated integral equation, one can solve \eqref{eq:BGTs}, for example, in $H^2(\T^2)$.
First, notice that the external force in \eqref{eq:BGTs} decays exponentially as $t\to T$, namely
\[ \Sobn{(1-\psi ^2)\psi |\tilde{R}_{\la}(t)|^2\tilde{R}_\la (t)-2\nabla\psi\nabla\tilde{R}_{\la}(t)-\De\psi\tilde{R}_{\la}(t)}{2}\lesssim e^{-\frac{\delta}{\la (T-t)}}\]
for some $\de >0$.
Thus, we can expect that the solution $v$ also decays exponentially as $t\to T$.

This decay of \emph{exponential} order is essential for the treatment of the linear terms $\psi ^2(\tilde{R}_\la ^2\bar{v}+|\tilde{R}_\la |^2v)$ in the fixed point argument.
To see this, we assume $\Lebn{v(t)}{2}\sim e^{-\frac{\mu}{\la (T-t)}}$ for some $\mu >0$ and consider the estimate for the $L^2$ norm of the Duhamel integral term, then
\begin{eqnarray*}
\Lebn{\int _t^Te^{i(t-s)\De}\left[ (\psi \tilde{R}_\la )^2v\right] (s)\,ds}{2}&\leq &\int _t^T \Lebn{\psi \tilde{R}_\la(s)}{\I}^2\Lebn{v(s)}{2}\,ds\\
&\sim &\int _t^T \frac{1}{\la ^2(T-s)^2}e^{-\frac{\mu}{\la (T-s)}}\,ds\\
&= &\frac{1}{\mu \la}e^{-\frac{\mu}{\la (T-t)}}\sim \frac{1}{\mu \la}\Lebn{v(t)}{2}.
\end{eqnarray*}
From the above estimate, the linear terms seem to be harmless (at least in $L^2$) whenever $\mu \la$ is sufficiently large.
In fact, under the assumption that $\la$ is sufficiently large, Burq, G\'erard, and Tzvetkov obtained an exponentially decaying solution by performing the fixed point argument in the space $C([0,T);H^2)$ with an appropriate weight in $t$ which grows exponentially as $t\to T$.
Note that any \emph{polynomial} decay in $T-t$ of solution will not be sufficient for us to close the fixed point argument.

Let us return to the Zakharov system \eqref{eq:Z} and take the same approach as \eqref{eq:NLS}.
Let $(P_{\la},N_{\la}):\R ^2\to \R ^2$ be a radially symmetric solution of
\begin{equation}\label{eq:GM}
\left\{
\begin{aligned}
&-\De P_{\la} + P_{\la} = N_{\la}P_{\la},\\
& \la^2(r^2\d_{rr}N_{\la}+6r\d_rN_{\la}+6N_{\la})-\De N_{\la}= \De |P_{\la}|^2,
\end{aligned}
\right.
\end{equation}
where $r=|x|$.
Then it is easy to check that $(u,n)$ defined as
\begin{eqnarray*}
u(t,x)&=&\frac{1}{\la(T-t)}e^{i\(\frac{1}{\la^2(T-t)}-\frac{|x|^2}{4(T-t)}\)}P_{\la}\(\frac{x}{\la(T-t)}\),\\
n(t,x)&=&\(\frac{1}{\la(T-t)}\)^2N_{\la}\(\frac{x}{\la(T-t)}\)
\end{eqnarray*}
is a solution of \eqref{eq:Z} in $\R^2$ which blows up as $t\to T$.
It was shown by Glangetas and Merle \cite{GM94} that for $\la >0$ sufficiently small the equation \eqref{eq:GM} actually has a solution with the following properties.
\begin{Prop}[\cite{GM94}]\label{prop:gm}
There exists a family of radially symmetric solutions $\{ (P_\la ,N_\la )\} _{0<\la \ll 1}$ to \eqref{eq:GM} such that $(P_\la ,N_\la )\to (Q,-Q^2)$ in $H^1(\R ^2) \times L^2(\R ^2)$ as $\la \to 0$.
Further, $(P_{\la}, N_{\la})\in H^k\times H^k$ for all $k\in \mathbb{N}\cup \{ 0\}$ and
\begin{equation*}
|P_{\la}^{(k)}(x)|\lesssim _ke^{-\de |x|},\quad |N_{\la}^{(k)}(x)|\lesssim _k \<x\>^{-(3+k)}
\end{equation*}
for some $\de >0$.
\end{Prop}
These blow-up solutions are similar to the solutions $\tilde{R}_\la$ of \eqref{eq:NLS}, but different from them in the following two points:\\
--- A large $\la$ is not allowed.\\
--- The solutions for the wave part decay only polynomially in $t$.

From time reversal symmetry, it suffices to consider solutions which blow up backward in time at $t=0$.
For small $\la >0$, let
\begin{eqnarray*}
\tilde{U}_{\la}(t,x)&:=&\frac{1}{\la t}e^{-i\(\frac{1}{\la^2t}-\frac{|x|^2}{4t}\)}P_{\la}\(\frac{x}{\la t}\),\\
\tilde{W}_{\la}(t,x)&:=&\frac{1}{\( \la t\)^2}N_{\la}\(\frac{x}{\la t}\)
\end{eqnarray*}
be the blow-up solution of (\ref{eq:Z}) in $\R^2$ constructed in \cite{GM94}.
With the cut function $\psi$ defined above, set
\begin{eqnarray*}
U_{\la}(t,x)&:=&\psi(x)\tilde{U}_{\la}(t,x),\\
W_{\la}(t,x)&:=&\psi(x)\tilde{W}_{\la}(t,x),\quad (t,x)\in \R\times [-\pi ,\pi ]^2\simeq \R \times \T^2 .
\end{eqnarray*}

We construct a blow-up solution of the form $(U_{\la}+u,W_{\la}+w)$, where $(u,w)$ does not blow up in the energy space as $t\rightarrow 0$.
Assuming that $(U_{\la}+u,W_{\la}+w)$ solves (\ref{eq:Z}), we obtain
\begin{equation}\label{eq:BGT}
\left\{
\begin{aligned}
&(i \d_t+\De) u = uw + (W_{\la}u+U_{\la}w)+\((\psi-1)\psi \tilde{U}_{\la}\tilde{W}_{\la}-2\nabla\psi\nabla\tilde{U}_{\la}-\De\psi\tilde{U}_{\la}\) ,\\
& (\d_{tt}-\De) w= \De |u|^2 +\De(\bar{U}_{\la}u+U_{\la}\bar{u})+F_\la ,
\end{aligned}
\right.
\end{equation}
where
\begin{align*}
&F_\la :=\De(|U_{\la}|^2)-\psi\De(|\tilde{U}_{\la}|^2)+2\nabla\psi\nabla\tilde{W}_{\la}+\De\psi\tilde{W}_{\la}\\
&=(\psi -1)\psi \De (|\tilde{U}_{\la}|^2)+2\nabla (\psi ^2)\nabla (|\tilde{U}_{\la}|^2)+\De (\psi ^2)|\tilde{U}_{\la}|^2+2\nabla\psi\nabla\tilde{W}_{\la}+\De\psi\tilde{W}_{\la}.
\end{align*}

Here, the first difficulty arises due to the fact that the external force term $F_\la(t)$ only decays polynomially.
We can thus expect only polynomial decay for $w$, then the same for $U_\la w$ in the Schr\"odinger part, then the same for $u$, which would not be enough to remove the singularity in the term $W_\la u$.

The idea to overcome this difficulty is to decompose $w$ into polynomially decaying part and exponentially decaying part.
Note that the slowly-decaying external force $F_\la (t)$ is restricted outside a ball $B(0,1)$.
The finite speed of propagation then suggests that the slowly-decaying part of $w(t)$ also vanishes around the origin for a short time.
Since $U_\la$ has an exponential decay outside a neighborhood of the origin, we can still expect the exponential decay for the product $U_\la w$.

To make this argument rigorous, let $Z_{\la ,a}(t,x)$ be the solution of the following inhomogeneous linear wave equation for $a\in \R$:
\begin{equation}\label{eq:inhomwave}
\left\{
\begin{aligned}
&(\d_{tt}-\De) Z_{\la ,a}= F_\la ,\\
&Z_{\la ,a}(0,x)=0,\quad \d_tZ_{\la ,a}(0,x)=a\psi (x)(1-\psi (x)).
\end{aligned}
\right.
\end{equation}
Note that $Z_{\la ,a}$ is explicitly written as
\begin{eqnarray*}
Z_{\la ,a}(t,x)&=&\frac{\sin (t|\nabla |)}{|\nabla|}\big[ a\psi (1-\psi )\big] (x)-\int _0^t\frac{\sin ((t-s)|\nabla |)}{|\nabla|}F_\la (s,x)\,ds\\
&=:&\Psi _a(t,x)+Z_\la (t,x).
\end{eqnarray*}
A direct calculation using Proposition~\ref{prop:gm} shows that $Z_{\la ,a}\in C^1([0,\I );H^k(\T ^2))$ for any $k\geq 0$ and
\begin{equation}\label{est:Z_la}
\sup _{0<t<T}\( t^{-1}\norm{Z_{\la ,a}(t)}_{H^k(\T^2)}+\norm{\d_tZ_{\la ,a}(t)}_{H^k(\T^2)}\) \lesssim _{k,T,\la ,a}1
\end{equation}
for $T>0$.
Moreover, both the external force term and the initial data in \eqref{eq:inhomwave} vanish in a ball $B(0,1)$, which together with the finite speed of propagation yields that $Z_{\la ,a}(t,x)\equiv 0$ on a ball $B(0,1/2)$ for $0<t<1/2$.
Actually any initial data that vanish around the origin may be sufficient for the fixed point argument, but we have selected the above ones for another reason to be mentioned below.

We shall construct a blow up solution of the form $(U_{\la}+u,W_{\la}+Z_{\la ,a}+z)$, where $(u,z)$ converges to $0$ as $t\rightarrow 0$, solving
\begin{equation}\label{eq:BGT2}
\left\{
\begin{aligned}
&(i \d_t+\De) u = uz + (W_{\la}+Z_{\la ,a})u+U_{\la}z\\&\qquad \qquad \qquad +\(U_\la Z_{\la ,a}+(\psi-1)\psi \tilde{U}_{\la}\tilde{W}_{\la}-2\nabla\psi\nabla\tilde{U}_{\la}-\De\psi\tilde{U}_{\la}\),\\
& (\d_{tt}-\De) z= \De |u|^2 +\De(\bar{U}_{\la}u+U_{\la}\bar{u}).
\end{aligned}
\right.
\end{equation}
We notice that the external force in the Schr\"odinger part decays exponentially as $t\to 0$.

It is easy to see that a solution $(u,z)$ to the above problem satisfies
\[ \int _{\T^2}z_t(t,x)\,dx=c\hat{z}_t(t,0)\equiv 0\]
for all $t$.
Thus, $|\nabla|^{-1}z_t$ can be defined by $\sum_{m\neq (0,0)}e^{imx}|m|^{-1}\hat{z}_t(t,m)$.
Set
\begin{eqnarray*}
r=z+i|\nabla|^{-1}z_t.
\end{eqnarray*}
Then, if $(u,z)$ solves \eqref{eq:BGT2}, $(u,r)$ satisfies
\begin{equation}\label{eq:KM}
\left\{
\begin{aligned}
&(i \d_t+\De) u = u\Re r + (W_{\la}+Z_{\la ,a})u+U_{\la}\Re r\\&\qquad \qquad \qquad +\(U_\la Z_{\la ,a}+(\psi-1)\psi \tilde{U}_{\la}\tilde{W}_{\la}-2\nabla\psi\nabla\tilde{U}_{\la}-\De\psi\tilde{U}_{\la}\),\\
& (i\d_t-|\nabla|)r = |\nabla| \( |u|^2 +\bar{U}_{\la}u+U_{\la}\bar{u}\) .
\end{aligned}
\right.
\end{equation}
Since $z$ is real valued, we can recover the solution of \eqref{eq:BGT2} from that of \eqref{eq:KM} by $z:=\Re r$.
In this case $z_t=|\nabla |(\Im r)$ holds, and $(z,z_t)\in C((0,T];L^2(\T ^2;\R)\times \hat{H}^{-1}(\T^2 ;\R))$ if and only if $r\in C((0,T];L^2(\T^2 ;\C))$.

We will construct a local solution to this problem \eqref{eq:KM} that decays exponentially in $H^1(\T^2)\times L^2(\T^2)$ as $t\to 0$.
\begin{Thm}\label{Thm:perturb}
For any $a\in \R$ and sufficiently small $\la >0$, there exists $T=T(\la ,a)>0$ such that the equation \eqref{eq:KM} has a solution $(u,r)\in C((0,T],H^1(\T^2)\times L^2(\T^2))$ which decays exponentially as $t\to 0$.
\end{Thm}
Here, a solution of \eqref{eq:KM} means a distributional solution of the associated integral equation of \eqref{eq:KM}.

Now, we admit Theorem \ref{Thm:perturb} for a moment and show Theorem~\ref{Thm:blowup}.

\begin{proof}[Proof of Theorem~\ref{Thm:blowup}]
Recall that we have replaced the forward blowup at $t=T$ with the backward blowup at $t=0$.
For given $M>||Q||_{L^2(\R^2)}^2$, we choose $\la >0$ so that $||P_\la ||_{L^2(\R ^2)}^2<M$, which is possible from Proposition~\ref{prop:gm}.
Next, in the following way, we choose $a\in \R$ so that $\( W_{\la}+Z_{\la ,a}\) _t\in \hat{H}^{-1}$.
We first notice that
\[ \hat{\Psi}_{a,t}(t,0)\equiv a[\psi (1-\psi )]\hat{\;}(0)\]
and $[\psi (1-\psi )]\hat{\;}(0)=c\int \psi (1-\psi )\,dx >0$.
We also see from the equation that $[(W_\la +Z_\la )_t]\hat{\;}(t,0)$ is conserved.
Then, we choose $a\in \R$ so that $[(W_\la +Z_\la )_t]\hat{\;}(t,0)+a[\psi (1-\psi )]\hat{\;}(0)=0$.
With this choice of $\la$ and $a$, we set our blow-up solution of \eqref{eq:Z} as $(u,n)=(U_{\la}+v,W_{\la}+Z_{\la ,a}+\Re r)$, where $(v,r)$ is the solution of \eqref{eq:KM} (with $u$ replaced by $v$) obtained in Theorem~\ref{Thm:perturb}.
Note that this solution belongs to the energy space.

It is easily verified by the $L^2$ conservation law for \eqref{eq:Z} and the monotone convergence theorem that
\begin{equation}\label{eq:mct}
\int _{\T^2}|u(t,x)|^2\,dx=\lim _{t\to 0}\int _{\R^2}|\psi (x)\tilde{U}_\la (t,x)|^2\,dx=\int _{\R^2}|P_\la (x)|^2\,dx<M.
\end{equation}
Hence, we have proved (i).

To prove (ii) we claim the following stronger estimates: for $0<t\ll 1$,
\begin{align}
||\nabla u(t)||_{L^2(\T^2)}&\sim t^{-1},\label{est:burate1}\\
||n(t)||_{L^2(\T^2)}&\sim t^{-1},\label{est:burate2}\\
||n_t(t)||_{\hat{H}^{-1}(\T^2)}&\sim t^{-1}.\label{est:burate3}
\end{align}

For \eqref{est:burate1} and \eqref{est:burate2} it is sufficient to consider the main parts $U_\la (t)$ and $W_\la (t)$, respectively.
Similarly to \eqref{eq:mct}, we have
\[ \lim _{t\to 0}(\la t)^2\int _{\T^2}|W_\la (t,x)|^2\,dx=\int _{\R^2}|N_\la (x)|^2\,dx>0,\]
which implies \eqref{est:burate2}.
Also, we have
\begin{align*}
&|\nabla U_\la (t,x)|^2\\
&=\left| \frac{\psi (x)}{(\la t)^2}e^{-i( \frac{1}{\la ^2t}-\frac{|x|^2}{4t})}(\nabla P_\la )(\frac{x}{\la t})+\frac{i\psi (x)x}{2\la t^2}e^{-i( \frac{1}{\la ^2t}-\frac{|x|^2}{4t})}P_\la (\frac{x}{\la t})+(\nabla \psi )\tilde{U}_\la \right| ^2.
\end{align*}
Since
\[ \lim _{t\to 0}(\la t)^2\int _{\T^2}\left| \frac{\psi (x)}{(\la t)^2}(\nabla P_\la )(\frac{x}{\la t})\right| ^2dx=\int _{\R^2}|\nabla P_\la (x)|^2\,dx\]
and
\begin{gather*}
\int _{\T^2}\left| \frac{\psi (x)x}{2\la t^2}P_\la (\frac{x}{\la t})\right| ^2dx\leq (\frac{\la}{2})^2\int _{\R^2}\left| \frac{1}{\la t}\frac{x}{\la t}P_\la (\frac{x}{\la t})\right| ^2dx=(\frac{\la}{2})^2||xP_\la ||_{L^2}^2,\\
\int _{\T^2}|(\nabla \psi )(x)\tilde{U}_\la (t,x)|^2\,dx\leq ||\nabla \psi ||_{L^{\infty}}^2||\tilde{U}_\la (t)||_{L^2}^2=||\nabla \psi ||_{L^{\infty}}^2||P_\la ||_{L^2}^2, 
\end{gather*}
we conclude that
\[ \lim _{t\to 0}(\la t)^2\int _{\T^2}|\nabla U_\la (t,x)|^2\,dx=\int _{\R^2}|\nabla P_\la (x)|^2\,dx>0,\]
which implies \eqref{est:burate1}.

For \eqref{est:burate3}, 
it suffices to consider $(W_\la +Z_{\la ,a})_t(t)$ instead of $n_t(t)$.
We see that
\begin{align*}
(W_\la )_t(t,x)&=-\frac{2\psi (x)}{\la ^2t^3}N_\la (\frac{x}{\la t})-\frac{\psi (x)x}{\la ^3t^4}(\nabla N_\la )(\frac{x}{\la t})=-\psi (x)\nabla \left[ \frac{1}{\la t^2}\frac{x}{\la t}N_\la (\frac{x}{\la t})\right] \\
&=(\nabla\psi )(x) \left[ \frac{1}{\la t^2}\frac{x}{\la t}N_\la (\frac{x}{\la t})\right] -\nabla \left[ \frac{\psi (x)}{\la t^2}\frac{x}{\la t}N_\la (\frac{x}{\la t})\right] .
\end{align*}
Note that the second term is in $\hat{H}^{-1}$, then
\begin{align*}
\norm{(W_\la +Z_{\la ,a})_t(t)}_{\hat{H}^{-1}}=&\norm{\nabla \left[ \frac{\psi (x)}{\la t^2}\frac{x}{\la t}N_\la (\frac{x}{\la t})\right]}_{\hat{H}^{-1}}\\
&+O\( \norm{\nabla\psi \cdot \left[ \frac{1}{\la t^2}\frac{x}{\la t}N_\la (\frac{x}{\la t})\right]}_{L^2}+||(Z_{\la ,a})_t(t)||_{L^2}\) .
\end{align*}
Since $\nabla \psi (x)\equiv 0$ for $|x|<1$, we have
\begin{align*}
t^2\int _{\T^2}\left| (\nabla\psi )(x) \left[ \frac{1}{\la t^2}\frac{x}{\la t}N_\la (\frac{x}{\la t})\right] \right| ^2dx&\leq t^2\int _{\R^2}\left| \frac{1}{\la ^{1/2}t^{3/2}}(\frac{|x|}{\la t})^{\frac{3}{2}}N_\la (\frac{x}{\la t})\right| ^2dx\\
&=\la t \| |x|^{\frac{3}{2}}N_\la \| _{L^2}^2~\to ~0\quad (t\to 0).
\end{align*}
Recalling the definition of the $\hat{H}^{-1}$ norm (Remark~\ref{hatH}), we also have
\[ \lim _{t\to 0}t\norm{\nabla \left[ \frac{\psi (x)}{\la t^2}\frac{x}{\la t}N_\la (\frac{x}{\la t})\right]}_{\hat{H}^{-1}}=\lim _{t\to 0}\norm{\frac{\psi (x)}{\la t}\frac{x}{\la t}N_\la (\frac{x}{\la t})}_{L^2(\T^2)}=\norm{xN_\la}_{L^2(\R^2)},\]
where the first equality follows from the fact that $\frac{\psi (x)}{\la t^2}\frac{x}{\la t}N_\la (\frac{x}{\la t}) \in G$ for all $t>0$, which is verified by observing that $\frac{\psi (x)}{\la ^2t^3}N_\la (\frac{x}{\la t})=:f(|x|)$ is spherically symmetric and $xf(|x|)=\nabla (\int _0^{|x|}rf(r)\,dr+C)$.
This concludes \eqref{est:burate3}.

(iii) follows from a similar argument to the proof of \eqref{est:burate1} and \eqref{est:burate2}.
For instance, 
\begin{align*}
&\int _{\T^2\setminus B(0,r)}\left| \frac{\psi (x)}{(\la t)^2}(\nabla P_\la )(\frac{x}{\la t})\right| ^2dx\leq \int _{\T^2}\left| \frac{\psi (x)(1-\psi (2x/r))}{\la tr}\frac{|x|}{\la t}(\nabla P_\la )(\frac{x}{\la t})\right| ^2dx,
\end{align*}
which goes to $0$ as $t\to 0$ by the dominated convergence theorem, giving the claim for $\nabla u$.
We make a similar argument for $n$ and obtain (iii).
\end{proof}

All we have to do is to solve \eqref{eq:KM}.
However, compared to the case of \eqref{eq:NLS}, there are two major difficulties left:
(i) the loss of one derivative in the equation, and (ii) how to control the linear term $W_\la u$ ($\sim (\la t)^{-2}u$) without assuming $\la$ to be large.

When we construct solutions in the space $H^k\times H^{k-1}$, the loss of one derivative appears in the Schr\"odinger part and prevents us from applying the usual fixed point argument.
We shall employ the method of parabolic regularization to overcome this issue.
This method is also helpful in treating another issue (ii), because the viscosity effect will ease the singularity of $W_\la$ and give an extra small factor $T^{0+}$ to the corresponding term in the estimate.
The details will be discussed in Section~\ref{regularization}.

What is the most important is then the a priori estimate for the approximate solutions constructed via the parabolic regularization.
We meet the difficulties (i) and (ii) here again.

If we use the standard energy estimate, we will have only the estimate of $\frac{d}{dt}\Sobn{u(t)}{k}^2$ in terms of $\Sobn{u(t)}{k}$ and $\Sobn{r(t)}{k}$, which forces us to assume one more regularity for $r(t)$.
To obtain the a priori estimate in $H^k\times H^{k-1}$, we shall introduce a ``modified energy.''
More precisely, we modify the standard energy (the $H^k\times H^{k-1}$ norm of solutions) with harmless terms so that in the estimate of the time derivative of them the term including $\nabla ^kr(t)$ will be canceled (see Section~\ref{modified} for details).
This approach was recently taken by Kwon \cite{Kw08} and by Segata \cite{Se} for the fifth order KdV equation and the fourth order nonlinear Schr\"odinger equations with a derivative in the nonlinear term, respectively.
Note that this kind of modification on energy has a lot of ideas in common with the concept of ``correction terms'' in the context of the $I$-method introduced in a series of papers by Colliander, Keel, Staffilani, Takaoka, Tao.

Concerning (ii), the fact $W_\la (t,x)\in \R$ will be essential.
For instance, when we derive the identity for $\frac{d}{dt}\Lebn{u(t)}{2}^2$ the term corresponding to $W_\la u$ will not appear.
Similarly, in the estimate of $\frac{d}{dt}\Lebn{\nabla ^ku(t)}{2}^2$, there will be the terms like
\begin{equation}\label{w}
\int _{\T^2}\nabla ^lW_\la (t)\nabla ^{k-l}u(t)\nabla ^k\bar{u}(t)\,dx
\end{equation}
for $l=1,2,\dots, k$, but the term corresponding to $l=0$ will vanish.
On the other hand, note that $\Lebn{\nabla ^lW_\la}{\I}\lesssim (\la t)^{-2-l}$.
Then, if $\Lebn{\nabla ^{k-l}u(t)}{2}$ has a decay faster than $t^l\Lebn{\nabla ^ku(t)}{2}$ for each $l=1,2,\dots,k$, we can obtain the extra small factor again and control \eqref{w} by shrinking the time interval.
We will actually construct solutions with such a property, by carefully choosing the weight function in the norm for fixed point argument.
The precise definition of the norm will be given in Section~\ref{regularization}.

\begin{Rem}
In \cite{OzT92}, Ozawa and Tsutsumi proved the local well-posedness of the initial value problem for the Zakharov system on $\R^d$ ($d=1,2,3$) in the space $H^2\times H^1\times L^2$.
They pointed out that the loss of derivative does not occur when the Zakharov system is considered as the system of equations for $\d_tu$ and $n$.
This technique may be a solution to our difficulty (i), but it seems difficult to settle another issue (ii) by this idea.
That is why we employ the method of parabolic regularization.
\end{Rem}


\section{Parabolic Regularization}\label{regularization}

We first look for solutions $(u_{\ep},r_{\ep})$ to a regularized equation
\begin{equation}\label{eq:KMep}
\left\{
\begin{aligned}
&(i \d_t+\De+i\ep\De^2) u = u\Re r + (W_{\la}+Z_{\la,a})u+U_{\la}\Re r\\&\qquad \qquad \qquad +\(U_\la Z_{\la,a}+(\psi-1)\psi \tilde{U}_{\la}\tilde{W}_{\la}-2\nabla\psi\nabla\tilde{U}_{\la}-\De\psi\tilde{U}_{\la}\),\\
& (i\d_t-|\nabla|+i\ep\De^2)r = |\nabla| \( |u|^2 +\bar{U}_{\la}u+U_{\la}\bar{u}\),
\end{aligned}
\right.
\end{equation}
for $\ep >0$ in the space
\[ X_{T_{\ep}}=\left\{ (u,r)\in C((0,T_\ep ];H^3(\T ^2))\times C((0,T_{\ep}];H^2(\T ^2))\;:\; \norm{(u,r)}_{X_{T_{\ep}}}<\I \right\},\]
\begin{eqnarray*}
\norm{(u,r)}_{X_{T_{\ep}}}&=&\sup _{t\in (0,T_{\ep}]} \mathcal{H}[u,r](t),\\
\mathcal{H}[u,r](t)^2&:=&\( e^{\frac{\mu}{2\la t}}\norm{u(t)}_{\dot{H}^3(\T ^2)}\) ^2+\(t^{-4}e^{\frac{\mu}{2\la t}}\norm{u(t)}_{L^2(\T ^2)}\) ^2\\
&&\quad +\( t^{-\frac{2}{3}}e^{\frac{\mu}{2\la t}}\norm{r(t)}_{\dot{H}^2(\T ^2)}\) ^2+\( t^{-\frac{10}{3}}e^{\frac{\mu}{2\la t}}\norm{r(t)}_{L^2(\T^2)}\) ^2 ,
\end{eqnarray*}
where $\mu >0$ is the constant to be given in Lemma~\ref{lem:inhomo}.

\begin{Rem}
Note that by a simple interpolation, we have
\begin{eqnarray*}
\sup_{t\in(0,T_{\ep}]}e^{\frac{\mu}{2\la t}}\(t^{-\frac{8}{3}}||u||_{H^1}+t^{-\frac{4}{3}}||u||_{H^2}+t^{-2}||r||_{H^1}\)\leq C||(u,r)||_{X_{T_{\ep}}}.
\end{eqnarray*}

\end{Rem}

We prepare several lemmas.
\begin{Lem}\label{lem:para}
Let $V_{\ep}(t)u_0$ be the solution of
\begin{eqnarray*}
(i \d_t+\De+i\ep\De^2) u = 0,\ u(0)=u_0.
\end{eqnarray*}
Then, we have
\begin{eqnarray*}
||\int_0^tV_{\ep}(t-s)u(s)\,ds||_{H^k}\lesssim \ep^{-\frac{l}{4}}\int_0^t(t-s)^{-\frac{l}{4}}||u||_{H^{k-l}}.
\end{eqnarray*}
\end{Lem}

\begin{proof}
From the definition of $V_{\ep}$, we have
\begin{eqnarray*}
\int_0^tV_{\ep}(t-s)u(s)\,ds=\int_0^t\sum_{m\in\Z^2}e^{imx-i|m|^2(t-s)-\ep|m|^4(t-s)}\hat{u}_m(s)\,ds.
\end{eqnarray*}
Therefore, we have
\begin{eqnarray*}
&&||\int_0^tV_{\ep}(t-s)u(s)\,ds||_{H^k}\\
&\leq &\int_0^t\(\sum_{m\in\Z^2}\<m\>^{2k}e^{-\ep|m|^4(t-s)}|\hat{u}_m(s)|^2\)^{\frac{1}{2}}\,ds\\
&\leq &\int_0^t\sup_{\tilde{m}\in\Z^2}\<\tilde{m}\>^{l}e^{-\ep|\tilde{m}|^4(t-s)}\(\sum_{m\in\Z^2}\<m\>^{2(k-l)}|\hat{u}_m(s)|^2\)^{\frac{1}{2}}\,ds\\
&\lesssim& \ep^{-\frac{l}{4}}\int_0^t(t-s)^{-\frac{l}{4}}||u||_{H^{k-l}},
\end{eqnarray*}
where we have used
\begin{equation*}
\sup_{\tilde{m}\in\Z^2}\<\tilde{m}\>^{l}e^{-\ep|\tilde{m}|^4(t-s)}\lesssim \ep^{-\frac{l}{4}}(t-s)^{-\frac{l}{4}}.\qedhere
\end{equation*}
\end{proof}

We estimate the $L^p$ norms of $\nabla^k U_{\la}$ and $\nabla^k W_{\la}$.
\begin{Lem}\label{lem:UW}
Let $k\geq 0$, $\la$ sufficiently small and $p\in[1,\infty]$.
Then we have
\begin{eqnarray*}
||\nabla^k U_{\la}||_{L^p}&\lesssim& (\la t)^{-k-1+\frac{2}{p}},\\
||\nabla^k W_{\la}||_{L^p}&\lesssim& (\la t)^{-k-2+\frac{2}{p}}.
\end{eqnarray*}
\end{Lem}

\begin{proof}
By direct calculation using the definition of $U_{\la}$, $W_{\la}$ and Proposition \ref{prop:gm}, we have the conclusion.
\end{proof}

We next estimate the inhomogeneous term of the Schr\"odinger part of (\ref{eq:KMep}).
\begin{Lem}\label{lem:inhomo}
For any $k\geq 0$, there exist $C=C(k)>0$ and $\mu >0$ independent of $0<\la,t<1$ such that
\[ \Sobn{U_\la (t)Z_{\la,a}(t)+(\psi-1)\psi \tilde{U}_{\la}(t)\tilde{W}_{\la}(t)-2\nabla\psi\nabla\tilde{U}_{\la}(t)-\De\psi\tilde{U}_{\la}(t)}{k}\leq Ce^{-\frac{\mu}{\la t}}.\]
\end{Lem}
\begin{proof}
Since all the functions are supported away from $0$, the above estimate follows easily from the properties of $\tilde{U}_\la$, $\tilde{W}_\la$, and \eqref{est:Z_la}.
\end{proof}

Now, we construct a solution of \eqref{eq:KMep}.
\begin{Prop}\label{prop:regularize}
Let $0<\ep \leq 1$, and let $\la>0$ sufficiently small so that \eqref{eq:Z} on $\R^2$ may have a blow-up solution $(\tilde{U}_{\la},\tilde{W}_{\la})$ defined above.
Then, there exists a unique solution $(u,r)$ of \eqref{eq:KMep} in $X_{T_{\ep}}$, where $T_{\ep}\sim_{\la} \ep^{\frac{3}{2}}$.
Further, we have $\mathcal{H}[u,r](t)\to 0$ as $t\to 0$.
\end{Prop}

\begin{proof}
Set
\begin{eqnarray*}
Q_{S,0}&:=&U_\la Z_{\la,a}+(\psi-1)\psi \tilde{U}_{\la}\tilde{W}_{\la}-2\nabla\psi\nabla\tilde{U}_{\la}-\De\psi\tilde{U}_{\la}\\
Q_{S,1}&:=& (W_{\la}+Z_{\la,a})u+U_{\la}\Re r\\
Q_{S,2}&:=&u\Re r\\
Q_{W,1}&:=&|\nabla| (\bar{U}_{\la}u+U_{\la}\bar{u})\\
Q_{W,2}&:=&|\nabla| (|u|^2),
\end{eqnarray*}
and
\begin{eqnarray*}
I_{S,j}(t)&:=&\int_0^tV_{\ep}(t-s)Q_{S,j}(s)\,ds,\quad j=0,1,2,\\
I_{W,j}(t)&:=&\int_0^tU_{\ep}(t-s)Q_{W,j}(s)\,ds,\quad j=1,2.
\end{eqnarray*}
Set
\begin{eqnarray*}
\Phi_{\ep}(u,r):=\(\sum_{j=0}^2I_{S,j}, \sum_{l=1}^2I_{W,l}\).
\end{eqnarray*}
It suffices to show that $\Phi_{\ep}$ is a contraction mapping in $X_{T_{\ep}}$.

{\bf Estimate for $I_{S,0}$.}
\begin{eqnarray*}
||I_{S,0}||_{H^3}&\leq & \int_0^t||Q_{S,0}||_{H^3}\,ds \lesssim \int_0^t e^{-\frac{\mu}{\la s}}\,ds\\
&\lesssim&\la t^2e^{-\frac{\mu}{\la t}}. 
\end{eqnarray*}
Therefore,
\begin{eqnarray*}
\sup_{t\in (0,T_{\ep}]}e^{\frac{\mu}{2\la t}}t^{-4}||I_{S,0}||_{H^3}\lesssim \la T_{\ep}^{-2}e^{-\frac{\mu}{2\la T_{\ep}}}.
\end{eqnarray*}

{\bf Estimate for $I_{S,1}$.}
\begin{eqnarray*}
&||\int_0^t V_{\ep}(t-s)W_{\la}(s)u(s)\,ds||_{L^2}\lesssim 
\ep^{-\frac{1}{4}}\int_0^t(t-s)^{-\frac{1}{4}}||W_{\la}u||_{H^{-1}}\,ds\\
& \lesssim \ep^{-\frac{1}{4}}\int_0^t (t-s)^{-\frac{1}{4}}||W_{\la}u||_{L^{1+}}\,ds\\
&\qquad\lesssim \ep^{-\frac{1}{4}}\int_0^t (t-s)^{-\frac{1}{4}}||W_{\la}||_{L^{2+}}||u||_{L^2}\,ds\\
&\ \ \qquad\qquad\qquad\lesssim  \ep^{-\frac{1}{4}}\la^{-1-}\int_0^t (t-s)^{-\frac{1}{4}}s^{3-}e^{-\frac{\mu}{2\la s}}\,ds||(u,r)||_{X_{T_{\ep}}}\\
&\qquad\qquad\qquad\qquad\quad\lesssim \ep^{-\frac{1}{4}}\la^{-1-}\(\int_0^{t-t^2}\cdots\,ds+\int_{t-t^2}^t\cdots\,ds\)||(u,r)||_{X_{T_{\ep}}},
\end{eqnarray*}
where we have used Lemma \ref{lem:para} for the first inequality and Lemma \ref{lem:UW} for the fourth inequality.
Now, by
\begin{eqnarray*}
\int_0^{t-t^2} (t-s)^{-\frac{1}{4}}s^{3-}e^{-\frac{\mu}{2\la s}}\,ds&\lesssim &t^{-\frac{1}{2}}\int_0^{t-t^2}s^{3-}e^{-\frac{\mu}{2\la s}}\,ds\\
&\lesssim&t^{\frac{9}{2}-}e^{-\frac{\mu}{2\la t}},
\end{eqnarray*}
and
\begin{eqnarray*}
\int_{t-t^2}^t (t-s)^{-\frac{1}{4}}s^{3-}e^{-\frac{\mu}{2\la s}}\,ds&\lesssim &t^{-3-}e^{-\frac{\mu}{2\la t}}\int_{t-t^2}^t(t-s)^{-\frac{1}{4}}\,ds\\
&\lesssim&t^{\frac{9}{2}-}e^{-\frac{\mu}{2\la t}},
\end{eqnarray*}
we have
\begin{eqnarray*}
||\int_0^t V_{\ep}(t-s)W_{\la}(s)u(s)\,ds||_{L^2}&\lesssim & \ep^{-\frac{1}{4}}\la^{-1-}e^{-\frac{\mu}{2\la t}}t^{\frac{9}{2}-}||(u,r)||_{X_{T_{\ep}}}.
\end{eqnarray*}
Therefore,
\begin{eqnarray*}
\sup_{t\in (0,T_{\ep}]}e^{\frac{\mu}{2\la t}}t^{-4}||\int_0^t V_{\ep,\al}(t-s)W_{\la}(s)u(s)\,ds||_{L^2}\lesssim \ep^{-\frac{1}{4}}\la^{-1-}T_{\ep}^{\frac{1}{2}-}||(u,r)||_{X_{T_{\ep}}}.
\end{eqnarray*}
Next, we estimate
\begin{eqnarray*}
&&||\int_0^t V_{\ep}(t-s)W_{\la}(s)u(s)\,ds||_{H^3}\lesssim  \ep^{-\frac{1}{2}}\int_0^t(t-s)^{-\frac{1}{2}}||W_{\la}(s)u(s)||_{H^1}\,ds\\
&\lesssim &\ep^{-\frac{1}{2}}\int_0^t(t-s)^{-\frac{1}{2}}\(||\nabla W||_{L^{2}} ||u||_{H^{1+}}+||W_{\la}||_{L^{2}} ||u||_{H^{2+}}\)\,ds\\
&\lesssim &\ep^{-\frac{1}{2}}\int_0^t(t-s)^{-\frac{1}{2}}\((\la s)^{-2}s^{\frac{8}{3}-}e^{-\frac{\mu}{2\la s}}+(\la s)^{-1}s^{\frac{4}{3}-}e^{-\frac{\mu}{2\la s}}\)\,ds||(u,r)||_{X_{T_{\ep}}}\\
&\lesssim& \ep^{-\frac{1}{2}}\la^{-2}t^{\frac{4}{3}-}e^{-\frac{\la}{2\mu t}}||(u,r)||_{X_{T_{\ep}}}.
\end{eqnarray*}
Therefore,
\begin{eqnarray*}
\sup_{t\in (0,T_{\ep}]}e^{\frac{\mu}{2\la t}}||\int_0^t V_{\ep}(t-s)W_{\la}(s)u(s)\,ds||_{H^3}\lesssim \ep^{-\frac{1}{2}}\la^{-2}T_{\ep}^{\frac{4}{3}-}||(u,r)||_{X_{T_{\ep}}}.
\end{eqnarray*}
We estimate the second term of $I_{S,1}$.
\begin{eqnarray*}
||\int_0^t V_{\ep}(t-s)Z_{\la,a}(s)u(s)\,ds||_{L^2}&\leq &\int_0^t||Z_{\la,a}(s)||_{H^3}||u||_{L^2}\,ds\\
&\lesssim & \la t^7 e^{-\frac{\mu}{2\la t}}||(u,r)||_{X_{T_{\ep}}}.
\end{eqnarray*}
Therefore,
\begin{eqnarray*}
\sup_{t\in (0,T_{\ep}]}t^{-4}e^{\frac{\mu}{2\la t}}||\int_0^t V_{\ep,\al}(t-s)Z_{\la,a}(s)u(s)\,ds||_{L^2}\lesssim \la T_{\ep}^3||(u,r)||_{X_{T_{\ep}}}.
\end{eqnarray*}
\begin{eqnarray*}
||\int_0^t V_{\ep,\al}(t-s)Z_{\la,a}(s)u(s)\,ds||_{H^3}&\leq &\int_0^t||Z_{\la,a}(s)||_{H^3}||u||_{H^3}\,ds\\
&\lesssim & \la t^3 e^{-\frac{\mu}{2\la t}}||(u,r)||_{X_{T_{\ep}}}
\end{eqnarray*}
Therefore,
\begin{eqnarray*}
\sup_{t\in (0,T_{\ep}]}e^{\frac{\mu}{2\la t}}||\int_0^t V_{\ep,\al}(t-s)Z_{\la,a}(s)u(s)\,ds||_{H^3}\lesssim \la T_{\ep}^3||(u,r)||_{X_{T_{\ep}}}
\end{eqnarray*}
We estimate the third term of $I_{S,1}$.
\begin{eqnarray*}
||\int_0^t V_{\ep}(t-s)U_{\la}(s)\Re r(s)\,ds||_{L^2}&\lesssim &\int_0^t(\la s)^{-1}s^{\frac{10}{3}}e^{-\frac{\mu}{2\la s}}\,ds||(u,r)||_{X_{T_{\ep}}}\\
&\lesssim &t^{\frac{13}{3}}e^{-\frac{\mu}{2\la t}}||(u,r)||_{X_{T_{\ep}}}
\end{eqnarray*}
Therefore,
\begin{eqnarray*}
\sup_{t\in (0,T_{\ep}]}t^{-4}e^{\frac{\mu}{2\la t}}||\int_0^t V_{\ep,\al}(t-s)U_{\la}(s)\Re r(s)\,ds||_{L^2} \lesssim T_{\ep}^{\frac{1}{3}}||(u,r)||_{X_{T_{\ep}}}.
\end{eqnarray*}

\begin{eqnarray*}
&&||\int_0^t V_{\ep}(t-s)U_{\la}(s)\Re r(s)\,ds||_{H^3}\lesssim \ep^{-\frac{1}{4}}\int_0^t(t-s)^{-\frac{1}{4}}||U_{\la}\Re r||_{H^2}\,ds\\
&&\lesssim \ep^{-\frac{1}{4}}\int_0^t(t-s)^{-\frac{1}{4}}\(||\nabla^2U_{\la}||_{L^{\infty}}||\Re r||_{L^2}+||U_{\la}||_{L^{\infty}}||\Re r||_{H^2}\)\,ds\\
&&\lesssim \ep^{-\frac{1}{4}}\int_0^t(t-s)^{-\frac{1}{4}}\((\la s)^{-3}s^{\frac{10}{3}}e^{-\frac{\mu}{2\la s}}+(\la s)^{-1}s^{\frac{2}{3}}e^{-\frac{\mu}{2\la s}}\)\,ds||(u,r)||_{X_{T_{\ep}}}\\
&&\lesssim \ep^{-\frac{1}{4}}\(\la^{-3}t^{\frac{11}{6}}+\la^{-1}t^{\frac{7}{6}}\)e^{-\frac{\mu}{2\la t}}||(u,r)||_{X_{T_{\ep}}}.
\end{eqnarray*}
Therefore,
\begin{eqnarray*}
\sup_{t\in (0,T_{\ep}]}e^{\frac{\mu}{2\la t}}||\int_0^t V_{\ep}(t-s)U_{\la}(s)\Re r(s)\,ds||_{H^3}\lesssim \ep^{-\frac{1}{4}}\la^{-3} T_{\ep}^{\frac{6}{7}}||(u,r)||_{X_{T_{\ep}}}.
\end{eqnarray*}
Collecting the above estimates, we have
\begin{eqnarray*}
\sup_{t\in (0,T_{\ep}]}e^{\frac{\mu}{2\la t}}\(||I_{S,1}||_{H^3}+t^{-4}||I_{S,1}||_{L^2}\)\lesssim \ep^{-\frac{1}{2}}\la^{-3}T^{\frac{1}{3}}||(u,r)||_{X_{T_{\ep}}}.
\end{eqnarray*}

{\bf Estimate for $I_{S,2}$}.
\begin{eqnarray*}
||\int_0^t V_{\ep}(t-s) u(s)\Re r(s)\,ds||_{H^3}&\lesssim & \ep^{-\frac{1}{4}}\int_0^t(t-s)^{-\frac{1}{4}}||u||_{H^2}||r||_{H^2}\,ds\\
&\lesssim &\ep^{-\frac{1}{4}}\int_0^t(t-s)^{-\frac{1}{4}}s^2e^{-\frac{\mu}{\la s}}\,ds||(u,r)||_{X_{T_{\ep}}}^2\\
&\lesssim &\ep^{-\frac{1}{4}}t^{\frac{7}{2}}e^{-\frac{\mu}{\la t}}||(u,r)||_{X_{T_{\ep}}}^2
\end{eqnarray*}
Therefore,
\begin{eqnarray*}
\sup_{t\in (0,T_{\ep}]}t^{-4}e^{\frac{\mu}{2\la t}}||\int_0^t V_{\ep,4}(t-s)u(s)\Re r(s)\,ds||_{H^3} \lesssim \ep^{-\frac{1}{4}}T_{\ep}^{-\frac{1}{2}}e^{-\frac{\mu}{2\la T_{\ep}}}||(u,r)||_{X_{T_{\ep}}}^2.
\end{eqnarray*}

{\bf Estimate for $I_{W,1}$}.
\begin{eqnarray*}
||I_{W,1}||_{L^2}&\lesssim & \int_0^t ||U_{\la} u ||_{H^1}\,ds\\
&\lesssim & \int_0^t ||\nabla U_{\la} ||_{L^{\infty}}||u||_{L^2}+||U_{\la}||_{L^{\infty}}||u||_{H^1}\,ds\\
&\lesssim & \int e^{-\frac{\mu}{2\la s}}\(\la^{-2}s^2+\la^{-1}s^{\frac{5}{3}}\)\,ds||(u,r)||_{X_{T_{\ep}}}\\
&\lesssim & \la^{-2}t^{\frac{11}{3}}e^{-\frac{\mu}{2\la t}}||(u,r)||_{X_{T_{\ep}}}.
\end{eqnarray*}
Therefore,
\begin{eqnarray*}
\sup_{t\in (0,T_{\ep}]}e^{\frac{\mu}{2\la t}}t^{-\frac{10}{3}}||I_{W,1}||_{L^2}\lesssim \la^{-2}T_{\ep}^{\frac{1}{3}}||(u,r)||_{X_{T_{\ep}}}.
\end{eqnarray*}
Next,
\begin{eqnarray*}
||I_{W,1}||_{H^2}&\lesssim & \int_0^t ||U_{\la} u ||_{H^3}\,ds\\
&\lesssim & \int_0^t ||\nabla^3 U_{\la} ||_{L^{\infty}}||u||_{L^2}+||U_{\la}||_{L^{\infty}}||u||_{H^3}\,ds\\
&\lesssim & \int e^{-\frac{\mu}{2\la s}}\((\la s)^{-4}s^4+(\la s)^{-1}\)\,ds||(u,r)||_{X_{T_{\ep}}}\\
&\lesssim & \la^{-3}t e^{-\frac{\mu}{2\la t}}||(u,r)||_{X_{T_{\ep}}}
\end{eqnarray*}
Therefore,
\begin{eqnarray*}
\sup_{t\in (0,T_{\ep}]}e^{\frac{\mu}{2\la t}}t^{-\frac{2}{3}}||I_{W,1}||_{H^2}\lesssim \la^{-3}T_{\ep}^{\frac{1}{3}}||(u,r)||_{X_{T_{\ep}}}.
\end{eqnarray*}

{\bf Estimate for $I_{W,2}$}.
\begin{eqnarray*}
||I_{W,2}||_{H^2}&\lesssim &\int_0^t ||u||_{H^2}||u||_{H^3}\,ds\\
&\lesssim &\int_0^ts^{\frac{4}{3}}e^{-\frac{\mu}{\la s}}\,ds||(u,r)||_{X_{T_{\ep}}}^2\\
&\lesssim &\la t^{\frac{10}{3}}e^{-\frac{\mu}{\la t}}||(u,r)||_{X_{T_{\ep}}}^2.
\end{eqnarray*}
Therefore, we have
\begin{eqnarray*}
\sup_{t\in (0,T_{\ep}]}e^{\frac{\mu}{2\la t}}t^{-\frac{10}{3}}||I_{W,2}||_{H^2}\lesssim \la e^{-\frac{\mu}{2\la T_{\ep}}}||(u,r)||_{X_{T_{\ep}}}^2.
\end{eqnarray*}

So far, we have shown that
\begin{eqnarray*}
||\Phi _\ep (u,r)||_{X_{T_\ep}}\lesssim _{\la ,\mu} \ep ^{\frac{1}{2}}T_\ep ^{\frac{1}{3}}||(u,r)||_{X_{T_{\ep}}}+(1+\ep ^{\frac{1}{4}}T_\ep ^{\frac{1}{6}})e^{-\frac{\mu}{4\la T_{\ep}}}||(u,r)||_{X_{T_{\ep}}}^2+e^{-\frac{\mu}{4\la T_{\ep}}}.
\end{eqnarray*}
In the same manner, we have
\begin{eqnarray*}
&&||\Phi _\ep (u_1,r_1)-\Phi _\ep (u_2,r_2)||_{X_{T_\ep}}\\
&&\lesssim _{\la ,\mu} \ep ^{\frac{1}{2}}T_\ep ^{\frac{1}{3}}||(u_1,r_1)-(u_2,r_2)||_{X_{T_{\ep}}}\\
&&\quad +(1+\ep ^{\frac{1}{4}}T_\ep ^{\frac{1}{6}})e^{-\frac{\mu}{4\la T_{\ep}}}\( ||(u_1,r_1)||_{X_{T_{\ep}}}+||(u_2,r_2)||_{X_{T_{\ep}}}\) ||(u_1,r_1)-(u_2,r_2)||_{X_{T_{\ep}}}.
\end{eqnarray*}
Now, we choose $T_\ep\sim _{\la ,\mu} \ep ^{\frac{3}{2}}$ so small that 
\begin{eqnarray*}
&&||\Phi _\ep (u,r)||_{X_{T_\ep}}\leq \frac{1}{2}||(u,r)||_{X_{T_{\ep}}}+C_0e^{-\frac{\mu}{4\la T_{\ep}}}\( ||(u,r)||_{X_{T_{\ep}}}^2+1\),\\
&&||\Phi _\ep (u_1,r_1)-\Phi _\ep (u_2,r_2)||_{X_{T_\ep}}\leq \frac{1}{2}||(u_1,r_1)-(u_2,r_2)||_{X_{T_{\ep}}}\\
&&\qquad +C_0e^{-\frac{\mu}{4\la T_{\ep}}}\( ||(u_1,r_1)||_{X_{T_{\ep}}}+||(u_2,r_2)||_{X_{T_{\ep}}}\) ||(u_1,r_1)-(u_2,r_2)||_{X_{T_{\ep}}}.
\end{eqnarray*}
for some $C_0=C_0(\la ,\mu )>0$.
We may also assume $(4C_0e^{-\frac{\mu}{4\la T_{\ep}}})^2\leq \frac{1}{2}$.
Then, these estimates are sufficient for us to conclude that $\Phi_{\ep}$ is a contraction mapping on a closed ball
\begin{equation}\label{largeball}
\left\{ (u,r)\in X_{T_{\ep}}\Big| \; ||(u,r)||_{X_{T_{\ep}}}\leq (8C_0)^{-1}e^{\frac{\mu}{4\la T_{\ep}}}\right\} ,
\end{equation}
and also on another ball
\begin{equation}\label{smallball}
\left\{ (u,r)\in X_{T_{\ep}}\Big| \; ||(u,r)||_{X_{T_{\ep}}}\leq 4C_0e^{-\frac{\mu}{4\la T_{\ep}}}\right\} ,
\end{equation}
from which the existence of a solution follows.

For the uniqueness in $X_{T_{\ep}}$, let us assume that $(u',r')\in X_{T_\ep}$ is another solution to \eqref{eq:KMep}.
Then, the uniqueness of solution in a ball \eqref{largeball} (with $T_\ep$ replaced by a smaller $t_0$) shows that these two solutions coincide with each other on a sufficiently small time interval $[0,t_0]$.
The coincidence on the whole interval $[0,T_\ep ]$ then follows from the uniqueness of solution for the initial value problem (see Remark~\ref{rem:IVP} below) and a standard continuity argument.

Finally, since the unique solution must coincide with a solution constructed in a small ball \eqref{smallball} (with $T_\ep$ replaced by a smaller $t_1$) , we have
\[ ||(u,r)||_{X_{t_1}}=\sup _{t\in (0,t_1]}\mathcal{H}[u,r](t)\to 0\qquad \text{as}\quad t_1\to 0,\]
which implies that $\mathcal{H}[u,r](t)\to 0$ as $t\to 0$.

This completes the proof of Proposition \ref{prop:regularize}.
\end{proof}

\begin{Rem}\label{rem:IVP}
It is much easier to solve the initial value problem
\begin{equation*}
\left\{
\begin{aligned}
&(i \d_t+\De+i\ep\De^2) u = u\Re r + (W_{\la}+Z_{\la,a})u+U_{\la}\Re r\\&\qquad \qquad \qquad +\(U_\la Z_{\la,a}+(\psi-1)\psi \tilde{U}_{\la}\tilde{W}_{\la}-2\nabla\psi\nabla\tilde{U}_{\la}-\De\psi\tilde{U}_{\la}\),\\
& (i\d_t-|\nabla|+i\ep\De^2)r = |\nabla| \( |u|^2 +\bar{U}_{\la}u+U_{\la}\bar{u}\),\quad (t,x)\in [t_0,t_0+T]\times \T ^2, \\
&(u(t_0,x),r(t_0,x))\in H^3(\T^2) \times H^2(\T^2) ,
\end{aligned}
\right.
\end{equation*}
starting from $t=t_0>0$.
By a standard argument, we obtain a unique solution $(u,r)\in C([t_0,t_0+T];H^3(\T^2)\times H^2(\T^2))$ with $T>0$ depending on
\[ ||u(t_0,\cdot )||_{H^3},\quad ||r(t_0,\cdot )||_{H^2},\quad \sup _{t\geq t_0}|| U_\la, W_\la, Z_{\la ,a}, \emph{external force terms}||_{H^3}.\]
\end{Rem}


\section{Modified Energy}\label{modified}
In this section we are devoted to giving an a priori estimate on the $X_T$ norm of solutions to \eqref{eq:KMep} which is uniform in $\ep \in (0,1]$.

Recall some inequalities that will be frequently used below.
\begin{Lem}[Gagliardo-Nirenberg]
Let $k>k_0\geq 0$ and $k_0\leq l\leq k-1$.
Then, we have
\[ \norm{\nabla ^l\phi}_{L^4(\T^2)}\lesssim \begin{cases} \norm{\phi}_{L^2(\T^2)}^{1-\frac{1}{2k}}\norm{\nabla ^k\phi}_{L^2(\T^2)}^{\frac{1}{2k}}+\norm{\phi}_{L^2(\T^2)},&\text{if } k_0=l=0,\\[10pt]
\norm{\nabla ^{k_0}\phi}_{L^2(\T^2)}^{\frac{1}{k-k_0}(k-l-\frac{1}{2})}\norm{\nabla ^k\phi}_{L^2(\T^2)}^{\frac{1}{k-k_0}(l-k_0+\frac{1}{2})},&\text{otherwise}.\end{cases}\]
\end{Lem}
\begin{proof}
By an easy interpolation argument, it suffices to show the following simple estimate:
\[ \norm{\phi}_{L^4(\T^2)}\lesssim \norm{\phi}_{L^2(\T^2)}^{\frac{1}{2}}\norm{\nabla \phi}_{L^2(\T^2)}^{\frac{1}{2}}\]
for $\phi \in H^1$ such that $\hat{\phi}(0)=0$.

Using the Hausdorff-Young inequality to go to the frequency space and applying the H\"older inequality, we have
\begin{eqnarray*}
&&\norm{\phi}_{L^4(\T^2)}\lesssim \tnorm{\hat{\phi}}_{\ell ^{4/3}(\Z ^2)}\leq \sum _{j=0}^{\I}\tnorm{\hat{\phi}}_{\ell ^{4/3}(2^j\leq |k|<2^{j+1})}\\
&\lesssim &\sum _{j<J_0}2^{\frac{j}{2}}\tnorm{\hat{\phi}}_{\ell ^2(\Z^2)}+\sum _{j\geq J_0}2^{-\frac{j}{2}}\tnorm{\widehat{|\nabla |\phi}}_{\ell ^2(\Z^2)},
\end{eqnarray*}
where $J_0\geq 0$ will be chosen in a moment.
If $\norm{\phi}_{L^2(\T^2)}>\norm{\nabla \phi}_{L^2(\T^2)}$, we choose $J_0=0$ to obtain the claim.
Otherwise, the last line is estimated by
\[ 2^{\frac{J_0}{2}}\norm{\phi}_{L^2(\T^2)}+2^{-\frac{J_0}{2}}\norm{\nabla \phi}_{L^2(\T^2)}.\]
We choose $J_0$ such that $2^{J_0}\leq \norm{\nabla \phi}_{L^2(\T^2)}/\norm{\phi}_{L^2(\T^2)}<2^{J_0+1}$, then the claim follows.
\end{proof}

\begin{Lem}[Young]
Let $a,b\geq 0$ and $1<p,q<\I$, $\frac{1}{p}+\frac{1}{q}=1$.
Then, we have
\[ ab \leq \frac{1}{p}a^p+\frac{1}{q}b^q.\]
\end{Lem}

Let us fix $\ep >0$ and a solution $(u,r)\in X_T$ to \eqref{eq:KMep} arbitrarily.
Write $\mathcal{H}(t)$ to denote $\mathcal{H}[u,r](t)$ for simplicity.
In the following, constants may depend on $\la,a,\mu$, but not on $\ep$ and $t$.
The desired a priori estimate will be given in Corollary~\ref{cor:apriori} at the end of this section.

We begin with a standard energy estimate.
\begin{Prop}\label{prop:energyest}
There exists $T_0>0$ independent of $\ep$ such that we have
\begin{equation*}
\begin{split}
&\frac{d}{dt}\mathcal{H}(t)^2+\frac{\mu}{\la t^2}\mathcal{H}(t)^2+2\ep e^{\frac{\mu}{\la t}}\hSobn{u(t)}{5}^2+2\ep t^{-\frac{4}{3}}e^{\frac{\mu}{\la t}}\hSobn{r(t)}{4}^2\\
\leq ~&2e^{\frac{\mu}{\la t}}\Im \int (u+U_\la )\nabla \De (\Re r)\nabla\De \bar{u}\quad +Ct^{-\frac{5}{3}}\mathcal{H}(t)^2+C(\mathcal{H}(t)+\mathcal{H}(t)^3) 
\end{split}
\end{equation*}
for any $t\in (0,\min \{ T_0,\,T\})$.
\end{Prop}
\begin{Rem}
(i) The first term in the right hand side is unfavorable, because it contains the third derivative of $r$ and thus cannot be controlled by $\mathcal{H}(t)$.
In order to cancel this term, we will introduce a modified energy later.

(ii) We do not neglect beneficial $\ep$ terms in the left hand side, which will be exploited later.
Also, the term $t^{-2}\mathcal{H}(t)^2$ will be used to absorb the diverging quadratic term in the right hand side.
\end{Rem}

\begin{proof}
Recall the definition of $\mathcal{H}(t)$.

{\bf Estimate for $t^{-8}e^{\frac{\mu}{\la t}}\Lebn{u(t)}{2}^2$}.
Using the equation, we have
\begin{eqnarray*}
&&\frac{d}{dt}\( t^{-8}e^{\frac{\mu}{\la t}}\Lebn{u(t)}{2}^2\) +\( \frac{8}{t}+\frac{\mu}{\la t^2}\)t^{-8}e^{\frac{\mu}{\la t}}\Lebn{u(t)}{2}^2\\
&=&2 t^{-8}e^{\frac{\mu}{\la t}}\Re \int \( -\ep \De ^2u-i\(-\De u +u\Re r+(W_\la +Z_{\la,a})u+U_\la \Re r+Q_{S,0}\) \)\bar{u}\\
&=&-2\ep t^{-8}e^{\frac{\mu}{\la t}}\hSobn{u(t)}{2}^2+2 t^{-8}e^{\frac{\mu}{\la t}}\Im \int \( U_\la (\Re r)\bar{u}+Q_{S,0}\bar{u}\) .
\end{eqnarray*}
We discard the first term and estimate the integral by
\begin{eqnarray*}
&&t^{-8}e^{\frac{\mu}{\la t}}\( \Lebn{U_\la}{\I}\Lebn{r}{2}\Lebn{u}{2}+\Lebn{Q_{S,0}}{2}\Lebn{u}{2}\)\\
&\lesssim &t^{-8}e^{\frac{\mu}{\la t}}\( t^{-1}t^{\frac{10}{3}}e^{-\frac{\mu}{2\la t}}t^4e^{-\frac{\mu}{2\la t}}\mathcal{H}(t)^2+e^{-\frac{\mu}{\la t}}t^4e^{-\frac{\mu}{2\la t}}\mathcal{H}(t)\)\\
&\lesssim &t^{-\frac{5}{3}}\mathcal{H}(t)^2+\mathcal{H}(t)
\end{eqnarray*}
for sufficiently small $t$.
Consequently, we obtain
\begin{equation}\label{energy1-1}
\frac{d}{dt}\( t^{-8}e^{\frac{\mu}{\la t}}\Lebn{u(t)}{2}^2\) +\frac{\mu}{\la t^2}t^{-8}e^{\frac{\mu}{\la t}}\Lebn{u(t)}{2}^2\leq C\( t^{-\frac{5}{3}}\mathcal{H}(t)^2+\mathcal{H}(t)\) .
\end{equation}

{\bf Estimate for $e^{\frac{\mu}{\la t}}\hSobn{u(t)}{3}^2$}.
We have
\begin{eqnarray*}
&&\frac{d}{dt}\( e^{\frac{\mu}{\la t}}\hSobn{u(t)}{3}^2\)+\frac{\mu}{\la t^2}e^{\frac{\mu}{\la t}}\Lebn{\nabla^3u(t)}{2}^2\\
&=&2e^{\frac{\mu}{\la t}}\Re \int \nabla\De\big( -\ep \De^2u\\[-5pt]
&&\qquad \qquad -i\(-\De u +(u+U_\la )\Re r+(W_\la +Z_{\la,a})u+Q_{S,0}\)\big) \nabla \De \bar{u}\\
&\leq &-2\ep e^{\frac{\mu}{\la t}}\hSobn{u(t)}{5}^2+2e^{\frac{\mu}{\la t}}\Im \int (u+U_\la )\nabla \De (\Re r)\nabla\De \bar{u}\\
&&+Ce^{\frac{\mu}{\la t}}\sum _{l=1}^3\( \Lebn{\nabla ^lu\nabla ^{3-l}r}{2}+\Lebn{\nabla ^lU_\la \nabla ^{3-l}r}{2}\)\Lebn{\nabla^3u}{2}\\
&&+Ce^{\frac{\mu}{\la t}}\sum _{l=1}^3\Lebn{\nabla ^l(W_\la +Z_{\la,a})\nabla ^{3-l}u}{2}\Lebn{\nabla^3u}{2}\\
&&+Ce^{\frac{\mu}{\la t}}\Lebn{\nabla ^3Q_{S,0}}{2}\Lebn{\nabla^3u}{2}.
\end{eqnarray*}
Note that the integral of $(W_\la +Z_{\la ,a})\nabla \De u\nabla \De \bar{u}$ vanishes.
We keep the first line and estimate the others.
For the second line,
\begin{align*}
&\sum _{l=1}^3\Lebn{\nabla ^lu\nabla ^{3-l}r}{2}\leq \Lebn{\nabla ^3u}{2}\Lebn{r}{\I}+\Lebn{\nabla ^2u}{4}\Lebn{\nabla r}{4}+\Lebn{\nabla u}{\I}\Lebn{\nabla ^2r}{2}\\
\lesssim &~\Lebn{\nabla ^3u}{2}\Sobn{r}{1+}+\Lebn{\nabla ^3u}{2}^{\frac{1}{2}}\Lebn{\nabla ^2u}{2}^{\frac{1}{2}}\Lebn{\nabla ^2r}{2}^{\frac{1}{2}}\Lebn{\nabla ^1r}{2}^{\frac{1}{2}}+\Sobn{u}{2+}\Lebn{\nabla ^2r}{2}\\[5pt]
\lesssim &~\( t^{2-}+t^{\frac{2}{3}}t^{\frac{1}{3}}t^{1}+t^{\frac{4}{3}-}t^{\frac{2}{3}}\) e^{-\frac{\mu}{\la t}}\mathcal{H}(t)^2\quad \sim t^{2-}e^{-\frac{\mu}{\la t}}\mathcal{H}(t)^2,\\
&\sum _{l=1}^3\Lebn{\nabla ^lU_\la \nabla ^{3-l}r}{2}\leq \sum _{l=1}^3\Lebn{\nabla ^lU_\la}{\I}\hSobn{r}{3-l}\lesssim \sum _{l=1}^3t^{-1-l}t^{\frac{4}{3}l-\frac{2}{3}}e^{-\frac{\mu}{2\la t}}\mathcal{H}(t)\\
&\qquad \lesssim ~t^{-\frac{4}{3}}e^{-\frac{\mu}{2\la t}}\mathcal{H}(t).
\end{align*}
For the third line, we have
\begin{align*}
&\sum _{l=1}^3\Lebn{\nabla ^l(W_\la +Z_{\la,a})\nabla ^{3-l}u}{2}\leq \sum _{l=1}^3\Lebn{\nabla ^l(W_\la +Z_{\la,a})}{\I}\hSobn{u}{3-l}\\[-5pt]
&\qquad \lesssim \sum _{l=1}^3t^{-2-l}t^{\frac{4}{3}l}e^{-\frac{\mu}{2\la t}}\mathcal{H}(t)\lesssim t^{-\frac{5}{3}}e^{-\frac{\mu}{2\la t}}\mathcal{H}(t).
\end{align*}
Finally,
\[ e^{\frac{\mu}{\la t}}\Lebn{\nabla ^3Q_{S,0}}{2}\Lebn{\nabla^3u}{2}\lesssim e^{-\frac{\mu}{2\la t}}\mathcal{H}(t).\]
Consequently, we obtain the following bound for small $t$:
\begin{equation}
\begin{split}
&\frac{d}{dt}\( e^{\frac{\mu}{\la t}}\hSobn{u(t)}{3}^2\)+\frac{\mu}{\la t^2}e^{\frac{\mu}{\la t}}\Lebn{\nabla^3u(t)}{2}^2+2\ep e^{\frac{\mu}{\la t}}\hSobn{u(t)}{5}^2\\
&\leq 2e^{\frac{\mu}{\la t}}\Im \int (u+U_\la )\nabla \De (\Re r)\nabla\De \bar{u}\quad +C\( \mathcal{H}(t)^3+t^{-\frac{5}{3}}\mathcal{H}(t)^2+\mathcal{H}(t)\) .
\end{split}
\end{equation}

{\bf Estimate for $t^{-\frac{20}{3}}e^{\frac{\mu}{\la t}}\Lebn{r(t)}{2}^2$}.
\begin{eqnarray*}
&&\frac{d}{dt}\( t^{-\frac{20}{3}}e^{\frac{\mu}{\la t}}\Lebn{r(t)}{2}^2\)+\( \frac{20}{3t}+\frac{\mu}{\la t^2}\)t^{-\frac{20}{3}}e^{\frac{\mu}{\la t}}\Lebn{r(t)}{2}^2\\
&=&2 t^{-\frac{20}{3}}e^{\frac{\mu}{\la t}}\Re \int \( -\ep \De ^2r-i|\nabla |\( r+ |u|^2 +\bar{U}_\la u+U_\la \bar{u}\) \)\bar{r}\\
&=&-2\ep t^{-\frac{20}{3}}e^{\frac{\mu}{\la t}}\hSobn{r(t)}{2}^2+2 t^{-\frac{20}{3}}e^{\frac{\mu}{\la t}}\Im \int \( |u|^2 +\bar{U}_\la u+U_\la \bar{u}\) |\nabla |\bar{r}.
\end{eqnarray*}
For the estimate of the integral, we use the following:
\begin{align*}
\Lebn{|u|^2 +\bar{U}_\la u+U_\la \bar{u}}{2}&\lesssim \Sobn{u}{1+}\Lebn{u}{2}+\Lebn{U_\la}{\I}\Lebn{u}{2}\\
&\lesssim t^{\frac{20}{3}-}e^{-\frac{\mu}{\la t}}\mathcal{H}(t)^2+t^3e^{-\frac{\mu}{2\la t}}\mathcal{H}(t).
\end{align*}
Discarding the $\ep$ term, we have
\begin{equation}
\frac{d}{dt}\( t^{-\frac{20}{3}}e^{\frac{\mu}{\la t}}\Lebn{r(t)}{2}^2\)+\frac{\mu}{\la t^2}t^{-\frac{20}{3}}e^{\frac{\mu}{\la t}}\Lebn{r(t)}{2}^2\leq C\( \mathcal{H}(t)^3+t^{-\frac{5}{3}}\mathcal{H}(t)^2\) 
\end{equation}
for small $t$.

{\bf Estimate for $t^{-\frac{4}{3}}e^{\frac{\mu}{\la t}}\hSobn{r(t)}{2}^2$}.
\begin{eqnarray*}
&&\frac{d}{dt}\( t^{-\frac{4}{3}}e^{\frac{\mu}{\la t}}\hSobn{r(t)}{2}^2\)+\( \frac{4}{3t}+\frac{\mu}{\la t^2}\) t^{-\frac{4}{3}}e^{\frac{\mu}{\la t}}\hSobn{r(t)}{2}^2\\
&=&2 t^{-\frac{4}{3}}e^{\frac{\mu}{\la t}}\Re \int \De \( -\ep \De^2 r-i|\nabla |\( r+ |u|^2 +\bar{U}_\la u+U_\la \bar{u}\) \)\De\bar{r}\\
&=&-2\ep t^{-\frac{4}{3}}e^{\frac{\mu}{\la t}}\hSobn{r(t)}{4}^2+2 t^{-\frac{4}{3}}e^{\frac{\mu}{\la t}}\Im \int \De |\nabla |\( |u|^2 +\bar{U}_\la u+U_\la \bar{u}\) \De\bar{r}.
\end{eqnarray*}
Noticing that
\begin{align*}
&\Lebn{\De |\nabla |\( |u|^2 +\bar{U}_\la u+U_\la \bar{u}\)}{2}\lesssim \Sobn{u}{3}^2+\sum _{l=0}^3\Lebn{\nabla ^lU_\la}{\I}\Lebn{\nabla ^{3-l}u}{2}\\
&\qquad \lesssim e^{-\frac{\mu}{\la t}}\mathcal{H}(t)^2+t^{-1}e^{-\frac{\mu}{2\la t}}\mathcal{H}(t),
\end{align*}
we obtain
\begin{equation}\label{energy1-4}
\begin{split}
&\frac{d}{dt}\( t^{-\frac{4}{3}}e^{\frac{\mu}{\la t}}\hSobn{r(t)}{2}^2\)+\frac{\mu}{\la t^2}t^{-\frac{4}{3}}e^{\frac{\mu}{\la t}}\hSobn{r(t)}{2}^2+2\ep t^{-\frac{4}{3}}e^{\frac{\mu}{\la t}}\hSobn{r(t)}{4}^2\\
&\qquad \leq C\( \mathcal{H}(t)^3+t^{-\frac{5}{3}}\mathcal{H}(t)^2\) 
\end{split}
\end{equation}
for small $t$.

We conclude the proof by collecting \eqref{energy1-1}--\eqref{energy1-4}.
\end{proof}

We introduce some additional terms into the energy so that the worst term will be cancelled out.
Define the modified energy $E(t)=E[u,r](t)$ by
\[ E(t):=\mathcal{H}(t)^2+2e^{\frac{\mu}{\la t}}\Re \int _{\T^2}(u+U_\la )\De \bar{u}\De \Re r+e^{\frac{\mu}{\la t}}\Sobn{u(t)}{1}^{10}. \]
In the energy estimate, the second term of $E(t)$ will produce an unfavorable term, which is exactly the same as that coming from $\mathcal{H}(t)^2$, except for the sign.
Therefore, we can eliminate the term including $\nabla ^3r(t)$ and close the energy estimate.
The last term in $E(t)$ ensures the positivity of the modified energy, as we will see in the next lemma.
\begin{Lem}\label{lem:modenergy}
we have
\[ E(t)\sim \mathcal{H}(t)^2+e^{\frac{\mu}{\la t}}\Sobn{u(t)}{1}^{10}\]
for any sufficiently small $t>0$.
\end{Lem}
\begin{proof}
By the H\"older inequality,
\begin{eqnarray*}
&&\left| 2e^{\frac{\mu}{\la t}}\Re \int (u+U_\la )\De \bar{u}\De \Re r\right| \\
&\leq &Ce^{\frac{\mu}{\la t}}\( \Lebn{u}{4}\Lebn{\De u}{4}\Lebn{\De r}{2}+t^{-1}\Lebn{\De u}{2}\Lebn{\De r}{2}\) .
\end{eqnarray*}
Using the Gagliardo-Nirenberg and the Young inequalities, we have
\begin{eqnarray*}
&&Ce^{\frac{\mu}{\la t}}\Lebn{u}{4}\Lebn{\De u}{4}\Lebn{\De r}{2}\\
&\leq &Ce^{\frac{\mu}{\la t}}\( \Lebn{u}{2}^{1/2}\Lebn{\nabla u}{2}^{1/2}+\Lebn{u}{2}\) \Lebn{\nabla u}{2}^{1/4}\Lebn{\nabla ^3u}{2}^{3/4}\Lebn{\De r}{2}\\
&\leq &Ce^{\frac{\mu}{\la t}}\Sobn{u}{1}^{5/4}\Lebn{\nabla ^3u}{2}^{3/4}\Lebn{\De r}{2}\\
&\leq &\frac{1}{2}e^{\frac{\mu}{\la t}}\( \Sobn{u}{1}^{5/4}\) ^8+\frac{1}{4}e^{\frac{\mu}{\la t}}\( \Lebn{\nabla ^3u}{2}^{\frac{3}{4}}\) ^{\frac{8}{3}}+Ct^{\frac{4}{3}}\cdot t^{-\frac{4}{3}}e^{\frac{\mu}{\la t}}\Lebn{\De r}{2}^2\\
&\leq &\frac{1}{2}e^{\frac{\mu}{\la t}}\Sobn{u}{1}^{10}+\frac{1}{4}\mathcal{H}(t)^2
\end{eqnarray*}
if $t>0$ is sufficiently small.
Also,
\[ Ct^{-1}e^{\frac{\mu}{\la t}}\Lebn{\De u}{2}\Lebn{\De r}{2}\leq Ct\cdot t^{-\frac{4}{3}}e^{\frac{\mu}{2\la t}}\Lebn{\De u}{2}\cdot t^{-\frac{2}{3}}e^{\frac{\mu}{2\la t}}\Lebn{\De r}{2}\leq \frac{1}{4}\mathcal{H}(t)^2\]
if $t>0$ is sufficiently small, which implies the claim.
\end{proof}

Now, we state the key estimate.
\begin{Prop}\label{prop:modenergyest}
There exists $T_0>0$ independent of $\ep$ such that we have
\[ \frac{d}{dt}E(t)+\ep e^{\frac{\mu}{\la t}}\hSobn{u(t)}{5}^2+\ep t^{-\frac{4}{3}}e^{\frac{\mu}{\la t}}\hSobn{r(t)}{4}^2\leq C(1+E(t))^3\]
for $t\in (0,\min \{ T_0,\,T\} )$.
\end{Prop}

\begin{proof}
We start the proof with the estimate for $e^{\frac{\mu}{\la t}}\Sobn{u(t)}{1}^{10}$,
\begin{eqnarray*}
&&\frac{d}{dt}\( e^{\frac{\mu}{\la t}}\Sobn{u(t)}{1}^{10}\) +\frac{\mu}{\la t^2}e^{\frac{\mu}{\la t}}\Sobn{u(t)}{1}^{10}=5e^{\frac{\mu}{\la t}}\Sobn{u(t)}{1}^8\frac{d}{dt}\( \Sobn{u(t)}{1}^2\) \\
&=&5e^{\frac{\mu}{\la t}}\Sobn{u(t)}{1}^8\cdot 2\Re \int \\
&&\< \nabla \> \( -\ep \De^2 u-i\(-\De u +(u+U_\la )\Re r+(W_\la +Z_{\la,a})u+Q_{S,0}\)\)\< \nabla \> \bar{u}\\
&\leq &Ce^{\frac{\mu}{\la t}}\Sobn{u(t)}{1}^8\Lebn{(u+U_\la )\Re r+(W_\la +Z_{\la,a})u+Q_{S,0}}{2}\Sobn{u}{2}.
\end{eqnarray*}
In the last inequality we have discarded the $\ep$ term.
Since we have
\begin{eqnarray*}
&&\Lebn{(u+U_\la )\Re r+(W_\la +Z_{\la,a})u+Q_{S,0}}{2}\\
&\lesssim &\( \Sobn{u}{1+}+t^{-1}\)\Lebn{r}{2}+t^{-2}\Lebn{u}{2}+\Lebn{Q_{S,0}}{2}\\
&\lesssim &t^{6-}e^{-\frac{\mu}{\la t}}\mathcal{H}(t)^2+t^2e^{-\frac{\mu}{2\la t}}\mathcal{H}(t)+e^{-\frac{\mu}{\la t}},
\end{eqnarray*}
the last line is estimated by $C\( 1+\Sobn{u(t)}{1}^{10}\)\( \mathcal{H}(t)^3+\mathcal{H}(t)\)$.
In particular,
\begin{equation}\label{energy2}
\frac{d}{dt}\( e^{\frac{\mu}{\la t}}\Sobn{u(t)}{1}^{10}\)\leq C\( 1+e^{\frac{\mu}{\la t}}\Sobn{u(t)}{1}^{10}\)\( \mathcal{H}(t)^3+\mathcal{H}(t)\) .
\end{equation}


Next, we estimate
\begin{align}
\frac{d}{dt}&\( 2e^{\frac{\mu}{\la t}}\Re \int (u+U_\la )\De \bar{u}\De \Re r\) \notag \\
=&-\frac{2\mu}{\la t^2}e^{\frac{\mu}{\la t}}\Re \int (u+U_\la )\De \bar{u}\De \Re r\label{energy3a}\\
&+2e^{\frac{\mu}{\la t}}\Re \int \d_tu \De \bar{u}\De \Re r\label{energy3b}\\
&+2e^{\frac{\mu}{\la t}}\Re \int \d_tU_\la \De \bar{u}\De \Re r\label{energy3c}\\
&+2e^{\frac{\mu}{\la t}}\Re \int (u+U_\la )\De \d_t\bar{u}\De \Re r\label{energy3d}\\
&+2e^{\frac{\mu}{\la t}}\Re \int (u+U_\la )\De \bar{u}\De \d_t \Re r .\label{energy3e}
\end{align}

{\bf Estimate for \eqref{energy3a} and \eqref{energy3c}} is easy.
Note that a direct calculation implies $\Lebn{\d_tU_\la (t)}{\I}\lesssim t^{-3}$.
We have
\begin{equation}\label{energy3-1}
|(\ref{energy3a})+(\ref{energy3c})|\lesssim e^{\frac{\mu}{\la t}}\( t^{-2}\Lebn{u}{\I}+t^{-3}\) \Lebn{\De u}{2}\Lebn{\De r}{2}\lesssim \mathcal{H}(t)^3+t^{-1}\mathcal{H}(t)^2.
\end{equation}

{\bf Estimate for \eqref{energy3b}}.
From the equation, \eqref{energy3b} becomes
\[ 2e^{\frac{\mu}{\la t}}\Re \int \!\Big[ -\ep \De ^2 u-i\(-\De u +(u+U_\la )\Re r+(W_\la +Z_{\la,a})u+Q_{S,0}\) \Big] \De \bar{u}\De \Re r.\]
Only the first term is a bit tricky, since it contains $\nabla ^4u$.
Using the Gagliardo-Nirenberg inequality followed by the Young, we evaluate it as 
\begin{align*}
2\ep e^{\frac{\mu}{\la t}}\left| \int \De ^2 u\De \bar{u}\De \Re r\right| &\leq 2\ep e^{\frac{\mu}{\la t}}\Lebn{\De ^2 u}{4}\Lebn{\De u}{4}\Lebn{\De r}{2}\\
&\leq C\ep e^{\frac{\mu}{\la t}}\Lebn{\nabla ^5u}{2}^{\frac{3}{4}}\Lebn{\nabla ^3u}{2}^{\frac{3}{4}}\Lebn{\nabla ^2u}{2}^{\frac{1}{2}}\Lebn{\De r}{2}\\
&\leq \frac{\ep}{2}e^{\frac{\mu}{\la t}}\hSobn{u(t)}{5}^2+C\ep e^{\frac{\mu}{\la t}}\Lebn{\nabla ^3u}{2}^{\frac{6}{5}}\Lebn{\nabla ^2u}{2}^{\frac{4}{5}}\Lebn{\De r}{2}^{\frac{8}{5}}\\
&\leq \frac{\ep}{2}e^{\frac{\mu}{\la t}}\hSobn{u(t)}{5}^2+C\mathcal{H}(t)^{\frac{18}{5}}.
\end{align*}
The next term with $\De u$ vanishes.
The other parts are treated as usual,
\begin{eqnarray*}
&&2e^{\frac{\mu}{\la t}}\left| \int \Big[ (u+U_\la )\Re r+(W_\la +Z_{\la,a})u+Q_{S,0}\Big] \De \bar{u}\De \Re r\right| \\
&\lesssim &e^{\frac{\mu}{\la t}}\Lebn{(u+U_\la )\Re r+(W_\la +Z_{\la,a})u+Q_{S,0}}{\I}\Lebn{\De u}{2}\Lebn{\De r}{2}\\
&\lesssim &\mathcal{H}(t)^4+\mathcal{H}(t)^3+\mathcal{H}(t)^2.
\end{eqnarray*}
Collecting them, we have
\begin{equation}
|(\ref{energy3b})|\leq \frac{\ep}{2}e^{\frac{\mu}{\la t}}\hSobn{u(t)}{5}^2+C\( \mathcal{H}(t)^4+\mathcal{H}(t)^2\) .
\end{equation}
The $\ep$ term will be absorbed into the similar one appearing in the energy estimate for $\mathcal{H}(t)^2$ (Proposition~\ref{prop:energyest}).

{\bf Estimate for \eqref{energy3d}}.
This is equal to
\begin{align*}
&2e^{\frac{\mu}{\la t}}\Re \int (u+U_\la )\De \Big[ -\ep \De ^2 \bar{u}\Big] \De \Re r\\
&+2e^{\frac{\mu}{\la t}}\Re \int (u+U_\la )\De \Big[ -i\De \bar{u}\Big] \De \Re r\\
&+2e^{\frac{\mu}{\la t}}\Re \int (u+U_\la )\De \Big[ i\( (\bar{u}+\bar{U}_\la) \Re r+(W_\la +Z_{\la,a})\bar{u}+\bar{Q}_{S,0}\) \Big] \De \Re r\\
&=:(\ref{energy3d}\text{a})+(\ref{energy3d}\text{b})+(\ref{energy3d}\text{c}).
\end{align*}

(\ref{energy3d}a) contains the highest derivative.
For this, we first reduce derivatives on $\bar{u}$ by an integration by parts, and then make an argument similar to the case \eqref{energy3b},
\begin{align*}
&2\ep e^{\frac{\mu}{\la t}}\left| \int (u+U_\la )\De ^3\bar{u}\De \Re r\right| \\
&\leq 2\ep e^{\frac{\mu}{\la t}}\left| \int \nabla (u+U_\la )\nabla \De ^2\bar{u}\De \Re r\right| +2\ep e^{\frac{\mu}{\la t}}\left| \int (u+U_\la )\nabla \De ^2\bar{u}\nabla \De \Re r\right| \\
&\leq 2\ep e^{\frac{\mu}{\la t}}\Lebn{\nabla \De ^2u}{2}\( \Lebn{\nabla (u+U_\la )}{\I}\Lebn{\De r}{2}+\Lebn{u+U_\la}{\I}\Lebn{\nabla \De r}{2}\) \\
&\leq \frac{\ep}{4}e^{\frac{\mu}{\la t}}\hSobn{u(t)}{5}^2+C\ep e^{\frac{\mu}{\la t}}\( \Lebn{\nabla (u+U_\la )}{\I}^2\Lebn{\De r}{2}^2+\Lebn{u+U_\la}{\I}^2\Lebn{\nabla \De r}{2}^2\) .
\end{align*}
Again, the first term will be absorbed.
However, the remaining parts are still beyond the control of $\mathcal{H}(t)$.
For these terms, we have to exploit another gain of $2\ep t^{-\frac{4}{3}}e^{\frac{\mu}{\la t}}\hSobn{r(t)}{4}^2$ in the energy estimate for $\mathcal{H}(t)^2$.
That's exactly why in the previous section we have also added the same amount of viscosity to the wave equation as to the Schr\"odinger part, in spite of no loss of derivative in the wave part.
We estimate them as follows:
\begin{eqnarray*}
&&C\ep e^{\frac{\mu}{\la t}}\Lebn{\nabla (u+U_\la )}{\I}^2\Lebn{\De r}{2}^2\\
&\leq &C\ep e^{\frac{\mu}{\la t}}\( \Sobn{u}{2+}+t^{-2}\) ^2\Lebn{\nabla ^4r}{2}\Lebn{r}{2}\\
&\leq &\frac{\ep}{4}t^{-\frac{4}{3}}e^{\frac{\mu}{\la t}}\hSobn{r(t)}{4}^2+C\ep t^{\frac{4}{3}}e^{\frac{\mu}{\la t}}\( \Sobn{u}{2+}+t^{-2}\) ^4\Lebn{r}{2}^2\\
&\leq &\frac{\ep}{4}t^{-\frac{4}{3}}e^{\frac{\mu}{\la t}}\hSobn{r(t)}{4}^2+C\( \mathcal{H}(t)^6+\mathcal{H}(t)^2\) ,
\end{eqnarray*}
\begin{eqnarray*}
&&C\ep e^{\frac{\mu}{\la t}}\Lebn{u+U_\la}{\I}^2\Lebn{\nabla \De r}{2}^2\\
&\leq &C\ep e^{\frac{\mu}{\la t}}\( \Sobn{u}{1+}+t^{-1}\) ^2\Lebn{\nabla ^4r}{2}\Lebn{\nabla ^2r}{2}\\
&\leq &\frac{\ep}{4}t^{-\frac{4}{3}}e^{\frac{\mu}{\la t}}\hSobn{r(t)}{4}^2+C\ep t^{\frac{4}{3}}e^{\frac{\mu}{\la t}}\( \Sobn{u}{1+}+t^{-1}\) ^4\Lebn{\nabla ^2r}{2}^2\\
&\leq &\frac{\ep}{4}t^{-\frac{4}{3}}e^{\frac{\mu}{\la t}}\hSobn{r(t)}{4}^2+C\( \mathcal{H}(t)^6+t^{-\frac{4}{3}}\mathcal{H}(t)^2\) .
\end{eqnarray*}
Consequently, we have
\[ |(\ref{energy3d}\text{a})|\leq \frac{\ep}{4}e^{\frac{\mu}{\la t}}\hSobn{u(t)}{5}^2+\frac{\ep}{2}t^{-\frac{4}{3}}e^{\frac{\mu}{\la t}}\hSobn{r(t)}{4}^2+C\( \mathcal{H}(t)^6+t^{-\frac{4}{3}}\mathcal{H}(t)^2\) .\]

By an integration by parts, (\ref{energy3d}b) is equal to
\[ -2e^{\frac{\mu}{\la t}}\Im \int \nabla (u+U_\la )\nabla \De \bar{u}\De \Re r-2e^{\frac{\mu}{\la t}}\Im \int (u+U_\la )\nabla \De \bar{u}\nabla \De \Re r.\]
The first term is easily estimated, while the second one is unfavorable because of $\nabla ^3r$ and no $\ep$.
In fact, it is this term that cancels with the similar term occurring in the energy estimate for $\mathcal{H}(t)^2$.
We have
\begin{align*}
(\ref{energy3d}\text{b})&\leq 2e^{\frac{\mu}{\la t}}\Lebn{\nabla (u+U_\la )}{\I}\Sobn{u}{3}\Sobn{r}{2}-2e^{\frac{\mu}{\la t}}\Im \int (u+U_\la )\nabla \De \bar{u}\nabla \De \Re r\\
&\leq C\( \mathcal{H}(t)^3+t^{-\frac{4}{3}}\mathcal{H}(t)^2\) -2e^{\frac{\mu}{\la t}}\Im \int (u+U_\la )\nabla \De \bar{u}\cdot \nabla \De \Re r.
\end{align*}

For the third one, we see that
\begin{align*}
|(\ref{energy3d}\text{c})|\leq ~&2e^{\frac{\mu}{\la t}}\left| \int (u+U_\la )\De \Big[ \bar{u}\Re r+\bar{Q}_{S,0}\Big] \De \Re r\right| \\
&+2e^{\frac{\mu}{\la t}}\left| \int u\De \Big[ \bar{U}_\la \Re r+(W_\la +Z_{\la,a})\bar{u}\Big] \De \Re r\right| \\
&+2e^{\frac{\mu}{\la t}}\left| \int U_\la \De \Big[ \bar{U}_\la \Re r+(W_\la +Z_{\la,a})\bar{u}\Big] \De \Re r\right| \\
\lesssim ~&e^{\frac{\mu}{\la t}}\Lebn{u+U_\la}{\I}\( \Sobn{u}{2}\Sobn{r}{2}+\Sobn{Q_{S,0}}{2}\) \Lebn{\De r}{2}\\
&+e^{\frac{\mu}{\la t}}\Lebn{u}{\I}\( \Sobn{U_\la}{2}\Sobn{r}{2}+\Sobn{W_\la +Z_{\la,a}}{2}\Sobn{u}{2}\) \Lebn{\De r}{2}\\
&+e^{\frac{\mu}{\la t}}\Lebn{U_\la}{\I}\Lebn{\De r}{2}\\
&\quad \times \sum _{l=0}^2\( \Lebn{\nabla ^lU_\la}{\I}\Lebn{\nabla ^{2-l}r}{2}+\Lebn{\nabla ^l(W_\la +Z_{\la,a})}{\I}\Lebn{\nabla ^{2-l}u}{2}\) \\
\lesssim ~&\mathcal{H}(t)^4+\mathcal{H}(t)^3+t^{-1}\mathcal{H}(t)^2+\mathcal{H}(t).
\end{align*}

Finally, we have
\begin{equation}
\begin{split}
(\ref{energy3d})\leq &~\frac{\ep}{4}e^{\frac{\mu}{\la t}}\hSobn{u(t)}{5}^2+\frac{\ep}{2}t^{-\frac{4}{3}}e^{\frac{\mu}{\la t}}\hSobn{r(t)}{4}^2\\
&-2e^{\frac{\mu}{\la t}}\Im \int (u+U_\la )\nabla \De \bar{u}\cdot \nabla \De \Re r\\
&+C\( \mathcal{H}(t)^6+t^{-\frac{4}{3}}\mathcal{H}(t)^2+\mathcal{H}(t) \) .
\end{split}
\end{equation}

{\bf Estimate for \eqref{energy3e}} is similar to that for \eqref{energy3d}.
We see that
\begin{align*}
(\ref{energy3e})=~&2e^{\frac{\mu}{\la t}}\Re \int (u+U_\la )\De \bar{u}\De \Big[ -\ep \De ^2\Re r+|\nabla |\Im r\Big] \\
=~&-2\ep e^{\frac{\mu}{\la t}}\Re \int \De \big[ (u+U_\la )\De \bar{u}\big] \De ^2\Re r\\
&-2e^{\frac{\mu}{\la t}}\Re \int \nabla \big[ (u+U_\la )\De \bar{u}\big] \nabla |\nabla |\Im r\\
=:\;& (\ref{energy3e}\text{a})+(\ref{energy3e}\text{b}).
\end{align*}

It is easy to bound (\ref{energy3e}b) by $C(\mathcal{H}(t)^3+t^{-1/3}\mathcal{H}(t)^2)$.
On the other hand,
\begin{align*}
|(\ref{energy3e}\text{a})|\leq ~&2\ep e^{\frac{\mu}{\la t}}\Lebn{\De \big[ (u+U_\la )\De \bar{u}\big]}{2}\Lebn{\De ^2\Re r}{2}\\
\leq ~&\frac{\ep}{2}t^{-\frac{4}{3}}e^{\frac{\mu}{\la t}}\hSobn{r}{4}^2+C\ep t^{\frac{4}{3}}e^{\frac{\mu}{\la t}}\Lebn{\De \big[ (u+U_\la )\De \bar{u}\big]}{2}^2.
\end{align*}
For the estimate of the last term, we see that
\begin{eqnarray*}
&&C\ep t^{\frac{4}{3}}e^{\frac{\mu}{\la t}}\Lebn{\De (u+U_\la )}{4}^2\Lebn{\De u}{4}^2\\
&\leq &C\ep t^{\frac{4}{3}}e^{\frac{\mu}{\la t}}\( \Lebn{\nabla ^3u}{2}^{\frac{1}{2}}\Lebn{\nabla ^2u}{2}^{\frac{1}{2}}+t^{-\frac{5}{2}}\) ^2\Lebn{\nabla ^5u}{2}\Lebn{u}{2}\\
&\leq &\frac{\ep}{12}e^{\frac{\mu}{\la t}}\hSobn{u(t)}{5}^2+C\ep t^{\frac{8}{3}}e^{\frac{\mu}{\la t}}\( \Lebn{\nabla ^3u}{2}^{\frac{1}{2}}\Lebn{\nabla ^2u}{2}^{\frac{1}{2}}+t^{-\frac{5}{2}}\) ^4\Lebn{u}{2}^2\\
&\leq &\frac{\ep}{12}e^{\frac{\mu}{\la t}}\hSobn{u(t)}{5}^2+C\( \mathcal{H}(t)^6+\mathcal{H}(t)^2\) ,
\end{eqnarray*}
\begin{eqnarray*}
&&2C\ep t^{\frac{4}{3}}e^{\frac{\mu}{\la t}}\Lebn{\nabla (u+U_\la )}{\I}^2\Lebn{\nabla \De ^2u}{2}^2\\
&\leq &2C\ep t^{\frac{4}{3}}e^{\frac{\mu}{\la t}}\( \Sobn{u}{2+}+t^{-2}\) ^2\Lebn{\nabla ^5u}{2}\Lebn{\nabla u}{2}\\
&\leq &\frac{\ep}{12}e^{\frac{\mu}{\la t}}\hSobn{u(t)}{5}^2+C\ep t^{\frac{8}{3}}e^{\frac{\mu}{\la t}}\( \Sobn{u}{2+}+t^{-2}\) ^4\Lebn{\nabla u}{2}^2\\
&\leq &\frac{\ep}{12}e^{\frac{\mu}{\la t}}\hSobn{u(t)}{5}^2+C\( \mathcal{H}(t)^6+\mathcal{H}(t)^2\) ,
\end{eqnarray*}
\begin{eqnarray*}
&&C\ep t^{\frac{4}{3}}e^{\frac{\mu}{\la t}}\Lebn{u+U_\la}{\I}^2\Lebn{\De ^2u}{2}^2\\
&\leq &C\ep t^{\frac{4}{3}}e^{\frac{\mu}{\la t}}\( \Sobn{u}{1+}+t^{-1}\) ^2\Lebn{\nabla ^5u}{2}\Lebn{\nabla ^3u}{2}\\
&\leq &\frac{\ep}{12}e^{\frac{\mu}{\la t}}\hSobn{u(t)}{5}^2+C\ep t^{\frac{8}{3}}e^{\frac{\mu}{\la t}}\( \Sobn{u}{1+}+t^{-1}\) ^4\Lebn{\nabla ^3u}{2}^2\\
&\leq &\frac{\ep}{12}e^{\frac{\mu}{\la t}}\hSobn{u(t)}{5}^2+C\( \mathcal{H}(t)^6+t^{-\frac{4}{3}}\mathcal{H}(t)^2\) .
\end{eqnarray*}
Hence, we have
\begin{equation}\label{energy3-4}
\begin{split}
|(\ref{energy3e})|\leq &~\frac{\ep}{2}t^{-\frac{4}{3}}e^{\frac{\mu}{\la t}}\hSobn{r(t)}{4}^2+\frac{\ep}{4}e^{\frac{\mu}{\la t}}\hSobn{u(t)}{5}^2+C\( \mathcal{H}(t)^6+t^{-\frac{4}{3}}\mathcal{H}(t)^2 \) .
\end{split}
\end{equation}

Finally, we collect \eqref{energy3-1}--\eqref{energy3-4} to obtain
\begin{equation}\label{energy3}
\begin{split}
&\frac{d}{dt}\( 2e^{\frac{\mu}{\la t}}\Re \int (u+U_\la )\De \bar{u}\De \Re r\) \\
&\leq \ep e^{\frac{\mu}{\la t}}\hSobn{u(t)}{5}^2+\ep t^{-\frac{4}{3}}e^{\frac{\mu}{\la t}}\hSobn{r(t)}{4}^2-2e^{\frac{\mu}{\la t}}\Im \int (u+U_\la )\nabla \De \bar{u}\cdot \nabla \De \Re r\\
&+C\( \mathcal{H}(t)^6+t^{-\frac{4}{3}}\mathcal{H}(t)^2+\mathcal{H}(t) \) ,
\end{split}
\end{equation}
and combine Proposition~\ref{prop:energyest} with \eqref{energy2}, \eqref{energy3} to obtain
\begin{align*}
&\frac{d}{dt}E(t)+\frac{\mu}{\la t^2}\mathcal{H}(t)^2+\ep e^{\frac{\mu}{\la t}}\hSobn{u(t)}{5}^2+\ep t^{-\frac{4}{3}}e^{\frac{\mu}{\la t}}\hSobn{r(t)}{4}^2\\
&\qquad \leq Ct^{-\frac{5}{3}}\mathcal{H}(t)^2+C\( 1+\mathcal{H}(t)^2+e^{\frac{\mu}{\la t}}\Sobn{u(t)}{1}^{10}\) ^3.
\end{align*}

At the end, we take $t>0$ sufficiently small and use Lemma~\ref{lem:modenergy} to conclude the proof.
\end{proof}

\begin{Cor}\label{cor:apriori}
There exist $T_0>0$ and $C_0>0$ independent of $\ep$ such that any solution $(u,r)\in X_T$ to \eqref{eq:KMep} on a time interval $(0,T]$ with $0<T<T_0$ satisfies $||(u,r)||_{X_T}\leq C_0$.
\end{Cor}

\begin{proof}
Since $(u,r)\in X_T$, we see from Proposition~\ref{prop:regularize} that $\mathcal{H}(t)\to 0$ and thus $0\leq E(t)\lesssim \mathcal{H}(t)^2+\mathcal{H}(t)^{10}\to 0$ as $t\to 0$. 
By Proposition~\ref{prop:modenergyest}, we have
\[ \frac{d}{dt}\( -\frac{1}{(1+E(t))^2}\) =\frac{2}{(1+E(t))^3}\frac{d}{dt}E(t)\leq 2C\]
for $0<t<T$, where $C>0$ is the constant appearing in Proposition~\ref{prop:modenergyest}.
Integrating it on $(0,t)$, we have
\[ 1-\frac{1}{(1+E(t))^2}\leq 2Ct.\]
Thus, we have
\[ E(t)\leq \frac{1}{(1-2Ct)^{1/2}}-1\]
for $0<t<(2C)^{-1}$.
In particular, we have $\mathcal{H}(t)^2\lesssim E(t)\leq 1$ for $0<t<3/(8C)$.
\end{proof}

We are now in a position to prove Theorem~\ref{Thm:perturb}.
We will use the Ascoli-Arzel\`a theorem (see, for instance, \cite{KNbook}) to obtain a solution of \eqref{eq:KM} from approximate solutions.
\begin{Thm}[Ascoli-Arzel\`a]\label{thm:AA}
Let $X$ be a compact Hausdorff space and $Y$ a metric space.
Then a subset $F$ of $C(X;Y)$ is compact in the topology of uniform convergence if and only if it is equicontinuous, pointwise relatively compact and closed.
\end{Thm}

\begin{proof}[Proof of Theorem~\ref{Thm:perturb}]
For any $0<\ep \leq 1$, the solution $(u^\ep ,r^\ep )$ of \eqref{eq:KMep} on $(0,T_\ep ]$ constructed in Proposition~\ref{prop:regularize} is, by the solvability of the initial value problem (Remark~\ref{rem:IVP}) and the a priori estimate (Corollary~\ref{cor:apriori}), uniquely extended to the solution on $(0,T_0]$ belonging to $X_{T_0}$ and satisfying $||(u^\ep ,r^\ep )||_{X_{T_0}}\leq C_0$.
Then, for any (sufficiently small) $\de _0>0$, the family of functions $\{ (u^\ep ,r^\ep )\} _{0<\ep \leq 1}$ is uniformly bounded in $C([\de _0,T_0];H^3\times H^2)$.
Therefore, the Rellich-Kondrachov compactness theorem (see \cite{Adams}, for instance) implies that it is pointwise relatively compact in $C([\de _0,T_0];H^1\times L^2)$.

Let $\de _0\leq t<s\leq T_0$ and $0<\ep \leq 1$.
Using the equation, we see that
\begin{align*}
&\Sobn{u^\ep (s)-u^\ep (t)}{1}\leq \int _t^s\Sobn{\d_t u^\ep (\tau )}{1}d\tau \\
&\leq \int _t^s\bigg( \Sobn{u^\ep (\tau )}{3}+\ep \Sobn{u^\ep (\tau )}{5}+\Sobn{u^\ep (\tau )r^\ep (\tau )}{1}\\
&\quad +\Sobn{(W_\la +Z_{\la ,a})(\tau )u^\ep (\tau )}{1}+\Sobn{U_\la (\tau )r^\ep (\tau )}{1}+\Sobn{Q_{S,0}(\tau )}{1}\bigg) \,d\tau .
\end{align*}
We apply the a priori estimate for $(u^\ep ,r^\ep )$ in $X_{T_0}$ to bound the above integral by $C_{\de _0}(s-t)$, except for the term $\ep \int _t^s ||u^\ep (\tau )||_{H^5}\,d\tau$.
To evaluate this one, we integrate the estimate in Proposition~\ref{prop:modenergyest} to obtain
\[ \ep \int _t^se^{\frac{\mu}{\la \tau}}\hSobn{u^\ep (\tau )}{5}^2\,d\tau \lesssim \int _t^s (1+E[u^\ep ,r^\ep ](\tau ))^3d\tau +\max _{\tau =t,s}|E[u^\ep ,r^\ep ](\tau )|\lesssim 1,\]
where we have used $E(t)\lesssim 1$ which was shown in the proof of Corollary~\ref{cor:apriori}.
Therefore, by the Cauchy-Schwarz inequality, we have
\[ \ep \int _t^s\hSobn{u^\ep (\tau )}{5}\,d\tau \leq \ep ^\frac{1}{2}(s-t)^{\frac{1}{2}}\( \ep \int _t^s\hSobn{u^\ep (\tau )}{5}^2\,d\tau \) ^{\frac{1}{2}}\lesssim (s-t)^{\frac{1}{2}},\]
thus obtain
\[ \Sobn{u^\ep (s)-u^\ep (t)}{1}\lesssim (s-t)^{\frac{1}{2}}.\]
In the same manner, we also have
\[ \Lebn{r^\ep (s)-r^\ep (t)}{2}\lesssim (s-t)^{\frac{1}{2}},\]
concluding that $\{ (u^\ep ,r^\ep )\} _{0<\ep \leq 1}$ is equicontinuous in $C([\de _0,T_0];H^1\times L^2)$.

Applying Theorem~\ref{thm:AA}, we find a sequence $\ep _n\to 0$ such that $(u^{\ep _n},r^{\ep _n})$ converges strongly in $C([\de _0,T_0];H^1\times L^2)$.
Since $\de _0>0$ is arbitrary, writing the limit as $(u,r)\in C((0,T_0];H^1\times L^2)$, we have the a priori estimate for $(u,r)$:
\[ \sup _{0<t<T_0}\( t^{-\frac{8}{3}}e^{\frac{\mu}{2\la t}}\Sobn{u(t)}{1}+t^{-\frac{10}{3}}e^{\frac{\mu}{2\la t}}\Lebn{r(t)}{2}\) \lesssim 1,\]
namely, $(u,r)(t)$ decays exponentially in $H^1\times L^2$ as $t\to 0$.

It then suffices to show that the limit $(u,r)$ constructed above satisfies the integral equation associated to \eqref{eq:KM} on $(\de _0,T_0)\times \T^2$ for any $\de _0>0$ in the sense of distribution.
Let $V_\al (t):=e^{it\De -\al t\De ^2}$ be as in Lemma~\ref{lem:para}.
Since $||V_\al (t)||_{H^s\to H^s}\leq 1$ for $\al \geq 0$ and $s\in \R$, we see that
\begin{align*}
&\norm{\int _{\de _0}^t V_{\ep _n} (t-t')[u^{\ep _n}r^{\ep _n}_R](t')\,dt'-\int _{\de _0}^t V_0(t-t')[u\Re r](t')\,dt'}_{L^{\I}((0,T_0);H^{-1})}\\
&\leq \int _{\de _0}^{T_0}\Sobn{[u^{\ep _n}r^{\ep _n}_R](t')-[u\Re r](t')}{-1}dt'\\
&\quad +\sup _{\de _0<t<T_0}\int _{\de _0}^{t}\Sobn{(V_{\ep _n}(t-t')-V_0(t-t'))[u\Re r](t')}{-1}dt'.
\end{align*}
The first line in the right hand side tends to $0$ as $n\to \I$ because of the convergence $(u^{\ep _n}-u,r^{\ep _n}-r)\to (0,0)$ in $H^1\times L^2$ and the estimate
\[ ||fg||_{H^{-1}(\T^2)}\lesssim ||f||_{H^1(\T^2)}||g||_{L^2(\T^2)},\]
which follows from the Sobolev embedding.
The second line also converges to $0$, since $V_\al (t):H^{-1}\to H^{-1}$ converges to $V_0(t)$ as $\al \to 0$ in the strong operator topology uniformly in $t\in (0,T_0)$.
Consequently, we have
\begin{align*}
&\iint _{(\de _0,T_0)\times \T^2}\int _{\de _0}^t V_{\ep_n}(t-t')[u^{\ep _n}r^{\ep _n}_R](t',x)\,dt'\,\overline{\varphi (t,x)}\,dxdt\\
&\quad \to \quad \iint _{(\de _0,T_0)\times \T^2}\int _{\de _0}^t V_0(t-t')[u\Re r](t',x)\,dt'\,\overline{\varphi (t,x)}\,dxdt
\end{align*}
as $n\to \I$ for any $\varphi \in C^\I _0((\de _0,T_0)\times \T^2)$.
It is easier to show the convergences for the nonlinearity $(W_\la +Z_{\la ,a})u+U_\la \Re r$.
Similarly, we see that
\[ V_\ep (t-\de _0)u^\ep (\de _0)\quad \to \quad V_0(t-\de _0)u(\de _0)\qquad \text{in}\quad C([\de _0,T_0];H^1)\]
as $\ep \to 0$.
We can show that $r$ satisfies the equation in the same manner.
\end{proof}

We conclude this section by explaining the idea for the multi-point blow-up problem.
By a space translation, we may assume that all of given $p$ points in $\T^2$ are actually in $(-\pi ,\pi )^2$.
Following the idea of Godet~\cite{Godet} for the case of \eqref{eq:NLS}, we set our blow-up solution $(u,n)$ to \eqref{eq:Z} as
\begin{gather*}
u:=\sum _{j=1}^pU_j+v,\qquad n:=\sum _{j=1}^pW_j+w,\\
U_j(t,x):=\psi _j(x)\tilde{U}_{\la}(t,x-x_j),\quad W_j(t,x):=\psi _j(x)\tilde{W}_{\la}(t,x-x_j),
\end{gather*}
where $0<\la \ll 1$, $\psi _j(x):=\psi (\frac{x-x_j}{R})$, and $R>0$ is taken to be sufficiently small so that the $p$ balls $B(x_j,2R)$, the support of $\psi _j$, may be mutually disjoint and $\cup _{j=1}^pB(x_j,2R)\subset (-\pi ,\pi )^2$.
Thanks to the support property that $\psi _i\psi _j\equiv 0$ for $1\leq i\neq j\leq p$, there is no interaction between different blow-up portions, and we may treat the multi-point blow-up problem just like the one-point blowup.
In fact, if $(u,n)$ solves \eqref{eq:Z}, then $(v,w)$ will be a solution of
\begin{equation*}
\left\{
\begin{aligned}
&(i \d_t+\De) v = vw + (\textstyle\sum _jW_jv+\sum _jU_jw)+F_1,\\
& (\d_{tt}-\De) w= \De |v|^2 +\De(\textstyle\sum _j\bar{U}_jv+\sum _jU_j\bar{v})+F_2 ,
\end{aligned}
\right.
\end{equation*}
where
\begin{align*}
F_1:=&\textstyle\sum _j(\psi _j-1)\psi _j\tilde{U}_{\la}(t,x-x_j)\tilde{W}_{\la}(t,x-x_j)\\
&-\textstyle\sum _j\( 2\nabla\psi _j\nabla\tilde{U}_{\la}(t,x-x_j)+\De\psi _j\tilde{U}_{\la}(t,x-x_j)\) ,\\
F_2:=&\textstyle\sum _j\( \De(|\psi _j\tilde{U}_{\la}(t,x-x_j)|^2)-\psi _j\De(|\tilde{U}_{\la}(t,x-x_j)|^2)\) \\
&+\textstyle\sum _j\( 2\nabla\psi _j\nabla\tilde{W}_{\la}(t,x-x_j)+\De\psi _j\tilde{W}_{\la}(t,x-x_j)\) ,
\end{align*}
in which there is no interacting term like $\tilde{U}_{\la}(t,x-x_i)\tilde{W}_{\la}(t,x-x_j)$.
Since $F_2$ vanishes around blow-up points, the solution $Z=Z_{\la ,a}$ to the problem
\begin{equation*}
\left\{
\begin{aligned}
&(\d_{tt}-\De) Z= F_2,\\
&Z(0,x)=0,\quad \d_tZ(0,x)=a\textstyle\sum _j\psi _j(x)(1-\psi _j(x))
\end{aligned}
\right.
\end{equation*}
also vanishes on a neighborhood of each $x_j$ for a small time and satisfies the estimate like \eqref{est:Z_la}.
Imitating the argument in section~\ref{formulation}, we consider blow-up solutions of the form $(\sum _jU_j+v,\sum _jW_j+Z+\Re r)$, with $(v,r)$ satisfying
\begin{equation*}
\left\{
\begin{aligned}
&(i \d_t+\De) v = v\Re r + (\textstyle\sum _jW_j+Z)v+\sum _jU_j\Re r+\sum _jU_jZ+F_1,\\
&(i\d_t- |\nabla |) r= |\nabla | \( |v|^2 +\textstyle\sum _j\bar{U}_jv+\sum _jU_j\bar{v}\) ,\\
&(v(t),r(t))\to (0,0)\quad \text{in}~H^1\times L^2.
\end{aligned}
\right.
\end{equation*}
Note that the external force term $\sum _jU_jZ+F_1$ decays exponentially as $t\to 0$, which allows us to show by following the estimates in section~\ref{regularization} that the regularized version of the above problem has an exponentially-decaying solution.
For the a~priori estimate, we introduce the modified energy
\[ E[v,r](t):=\mathcal{H}[v,r](t)^2+2e^{\frac{\mu}{\la t}}\Re \int _{\T^2}(v+\textstyle\sum\limits _{j=1}^pU_j)\De \bar{v}\De \Re r+e^{\frac{\mu}{\la t}}\Sobn{v(t)}{1}^{10}\]
and just follow the proof in this section.
Consequently, we can construct a solution $(v,r)$ of the above problem, which gives a finite time blow-up solution $(u,n)$ of \eqref{eq:Z}.
This solution $(u,n)$ belongs to the energy class if we choose appropriate $a\in \R$.

Similarly to the proof of Theorem~\ref{Thm:blowup} in section~\ref{formulation}, we can verify that
\begin{gather*}
\int _{\T^2}|u(t,x)|^2\,dx\equiv p\int _{\R^2}|P_\la (x)|^2\,dx,\\
\lim _{t\to 0}(\la t)^2\int _{\T^2}|\nabla u(t,x)|^2\,dx=p\int _{\R^2}|\nabla P_\la (x)|^2\,dx,\\
\lim _{t\to 0}(\la t)^2\int _{\T^2}|n(t,x)|^2\,dx=p\int _{\R^2}|N_\la (x)|^2\,dx,
\end{gather*}
and
\[ \lim _{t\to 0}t^2\norm{n_t(t)}_{\hat{H}^{-1}(\T^2)}^2=p\norm{xN_\la}_{L^2(\R^2)}^2,\]
thus obtaining (i) and (ii) of Corollary~\ref{cor:blowup}.
Proof for (iii) is also a trivial modification of the one-point blow-up case.


\section{$L^2$ concentration and nonexistence of minimal mass blow-up solution}\label{sec:concentration}
In this section, we prove Theorem \ref{thm:concentration} and Corollary \ref{cor:nonexistence}.
We first observe that if the solution $(u,n)$ of \eqref{eq:Z} blows up in finite time, then $||\nabla u||_{L^2}$ must diverge.
\begin{Lem}\label{lem:bupu}
Suppose $(u,n)$ be a solution of \eqref{eq:Z} which blows up at $T\in(0,\infty)$.
Then, $||\nabla u(t)||_{L^2}\to \infty$ as $t\to T$.
\end{Lem}

\begin{proof}
First, by the local well-posedness and conservation of $||u||_{L^2}$, we have
\begin{eqnarray}\label{eq:bupcond}
||\nabla u(t)||_{L^2}+||n(t)||_{L^2}+||v(t)||_{L^2}\to \infty,\ \mathrm{as}\ t\to T,
\end{eqnarray}
where $v=-\nabla^{-1}\(\d_t n-\hat{n}_1(0)\)$.
Since the energy $\mathcal{E}$ satisfies
\begin{eqnarray*}
\frac{d}{dt}\mathcal{E}(u(t),n(t),v(t))=\int_{\T^2}\hat{n}_1(0)(n(t)+|u(t)|^2)\lesssim 1+ \mathcal{E}(u(t),n(t),v(t)),
\end{eqnarray*}
the Gronwall inequality imply
\begin{eqnarray*}
\mathcal{E}(u(t),n(t),v(t))\lesssim_T 1,\ t\in[0,T).
\end{eqnarray*}
Now, suppose there exists $t_n\to T$ such that $||\nabla u(t_n)||_{L^2}\lesssim 1$.
Then, we have
\begin{eqnarray*}
1&\gtrsim&\mathcal{E}(u(t_n),n(t_n),v(t_n))\\
&=&||\nabla u(t_n)||_{L^2}^2-\frac{1}{2}||u(t_n)||_{L^4}^4+\frac{1}{2}\int_{\T^2}\(n(t_n)+|u(t_n)|^2\)^2+\frac{1}{2}||v(t_n)||_{L^2}^2\\
&\gtrsim & -1 + \frac{1}{2}||v(t_n)||_{L^2}^2.
\end{eqnarray*}
Therefore, we see $||v(t_n)||_{L^2}^2$ is bounded.
Further, by the H\"older inequality, we have
\begin{eqnarray*}
1&\gtrsim&\mathcal{E}(u(t_n),n(t_n),v(t_n))\\
&\gtrsim& 1+ \frac{1}{2}||n(t_n)||_{L^2}\(||n(t_n)||_{L^2}-1\).
\end{eqnarray*}
Therefore, we have $||n(t_n)||_{L^2}^2$ bounded.
However, this contradicts with \eqref{eq:bupcond}.
Therefore, we have the conclusion.
\end{proof}

We next introduce a modification of the well known concentration compactness lemma.
\begin{Prop}[Modification of Proposition 1.7.6 of \cite{CazSemi}]\label{prop:cc}
Let $\la_n\to \infty$ and set $\T_n^2:=(\R/2\pi\la_n\Z)^2$.
Let $u_n\in H^1(\T_n^2)$, $||u_n||_{L^2(\T_n^2)}^2\to M$, $||\nabla u_n||_{L^2(\T_n^2)}\lesssim 1$.
Set
\begin{eqnarray*}
\mu:=\lim_{R\to\infty}\liminf_{n\to\infty}\sup_{y\in\T_n^2}\int_{|x-y|<R}|u_n(x)|^2\,dx.
\end{eqnarray*}
Then, there exists a subsequence $\{u_k\}\subset \{u_n\}$ with the following properties.
\begin{enumerate}
\item[$(\mathrm{i})$]
If $\mu=0$, then $\int_{\T_k^2}|u_k|^4\to 0$ as $k\to 0$.
\item[$(\mathrm{ii})$]
For arbitrary $l\geq 1$, there exist $v_{j,k}$ for $j=1,\cdots,l$ and $w_{l,k}$ such that
\begin{itemize}
\item[$(\mathrm{a})$]
For $j=1,\cdots,l$, there exists $y_{j,k}$ such that $\mathrm{supp}v_{j,k}\subset \{x\in \T_k^2 | |x-y_{j,k}|<\la_k/4\}$.
Further, considering $v_{j,k}$ as a function on $\R^2$, we have $v_{j,k}\rightharpoonup v_j$ weakly in $H^1(\R^2)$ and $v_{j,k}\to v_j$ in $L^2\cap L^4(\R^2)$ as $k\to \infty$.
\item[$(\mathrm{b})$]
$\mathrm{supp}v_{j,k}$ and $\mathrm{supp}w_{l,k}$ are pairwise disjoint.
Further, $\sum_{j=1}^l|v_{j,k}|+|w_{l,k}|\leq |u_k|$ and $\sum_{j=1}^l||v_k||_{H^1}+||w_k||_{H^1}\lesssim_l ||u_k||_{H^1}$.
\item[$(\mathrm{c})$]
$||v_{j,k}||_{L^2}^2\to \mu_j$ and $||w_{l,k}||_{L^2}^2\to M-\sum_{j=1}^l\mu_j$, where $\mu_1=\mu$ and $\mu_j=\lim_{R\to\infty}\liminf_{n\to\infty}\sup_{y\in\T_k^2}\int_{|x-y|<R}|w_{j-1,k}(x)|^2\,dx$ for $j\geq 2$.
Further, $\mu_j$ is monotonically nonincreasing and $\mu_l\to 0$ as $l\to \infty$.
\item[$(\mathrm{d})$]
$\int_{\T_k^2}|u_k|^r-\sum_{j=1}^l|v_{j,k}|^r-|w_{l,k}|^r\,dx\to 0$, $r=2,4$.
\item[$(\mathrm{e})$]
$\liminf_{k\to\infty}\int_{\T_k^2}|\nabla u_k|^2-\sum_{j=1}^l|\nabla v_k|^2-|\nabla w_k|^2\,dx\geq 0.$
\item[$(\mathrm{f})$]
$\int_{\T_k^2}|w_{l,k}|^4\lesssim \mu_l$ for sufficiently large $k$.
\end{itemize}
\end{enumerate}

\end{Prop}

We are now in a position to prove Theorem \ref{thm:concentration}.
\begin{proof}[Proof of Theorem \ref{thm:concentration}]
Let $(u,n)$ be a solution of \eqref{eq:Z} and suppose that it blows up at time $T\in(0,\infty)$.
By Lemma \ref{lem:bupu}, we have $||\nabla u(t)||_{L^2(\T^2)}\to\infty$ as $t\to T$.
Set
\begin{eqnarray*}
\mathcal{E}_0(u,n)&:=& \int_{\T^2}|\nabla u|^2+\frac{1}{2}\int_{\T^2}n^2+\int_{\T^2}n|u|^2,\\
\tilde{\mathcal{E}}(u)&:=&\int_{\T^2}|\nabla u|^2-\frac{1}{2}\int_{\T^2}|u|^4.
\end{eqnarray*}
Then, we have
\begin{eqnarray*}
\tilde{\mathcal{E}}(u)\leq \tilde{\mathcal{E}}(u)+\frac{1}{2}\int_{\T^2}(|u|^2+n)^2=\mathcal{E}_0(u,n).
\end{eqnarray*}
Take $t_n\to T$ and set $\la_n:=||\nabla u(t_n)||_{L^2(\T^2)}$.
Then, we have $\la_n\to\infty$. 
Set
\begin{eqnarray*}
U_n(x)&:=&\la_n^{-1}u(t_n,\la_n^{-1}x),\\
N_n(x)&:=&\la_n^{-2}n(t_n,\la_n^{-1}x).
\end{eqnarray*}
Then, we have $||U_n||_{L^2(\T_n^2)}=||u_0||_{L^2(\T^2)}$ and $||\nabla U_n||_{L^2(\T_n^2)}=1$,
where $\T_n^2=(\R/2\pi\la_n\Z)^2$.
Further, since $\mathcal{E}_0\leq \mathcal{E}\lesssim 1$, we have
\begin{eqnarray*}
\mathcal{E}_{0n}(U_n,N_n)=\la_n^{-2}\mathcal{E}_0(u(t_n),n(t_n))\to 0,
\end{eqnarray*}
where $\mathcal{E}_{0n}(U_n,N_n)=\int_{\T_n^2}|\nabla U_n|^2+\frac{1}{2}N_n^2+N_n|U_n|^2$.
Therefore, we have
\begin{eqnarray*}
\limsup_{n\to\infty}\tilde{\mathcal{E}}_n(U_n)\leq \limsup_{n\to\infty}\mathcal{E}_{0n}(U_n,N_n)=0,
\end{eqnarray*}
where $\tilde{\mathcal{E}}_n(U_n)=\int_{\T_n^2}|\nabla U_n|^2-\frac{1}{2}|U_n|^4$.

We prove the theorem by contradiction.
So, we assume there exist $R_0$, $\de_0$ and $n_0$ such that for $n\geq n_0$, we have
\begin{eqnarray*}
\sup_{y\in\T^2}\int_{|x-y|<R_0}|u(t_n,x)|^2\,dx\leq M_0-\de_0.
\end{eqnarray*}
We now apply Proposition \ref{prop:cc} to $U_n$.
First,
\begin{eqnarray*}
\mu&=&\lim_{R\to \infty}\liminf_{n\to \infty}\sup_{y\in\T_n^2}\int_{|x-y|<R}|U_n(x)|^2\,dx\\
&=&\lim_{R\to \infty}\liminf_{n\to \infty}\sup_{y\in\T_n^2}\int_{|x-y|<R/\la_n^2}|u(t_n,x)|^2\,dx\\
&\leq &M_0-\de_0.
\end{eqnarray*}
Therefore, we have $\mu\in[0,M_0-\de_0]$.
Let $U_k$ be the subsequences given by Proposition \ref{prop:cc}.
Now, suppose $\mu=0$, then we have
\begin{eqnarray*}
0&=& \liminf_{k\to\infty}\int_{\T_n^2}|U_k|^4\,dx\\
&=& 2\liminf_{k\to \infty} \(\int_{\T_n^2}|\nabla U_k|^2\,dx-\tilde{\mathcal{E}}_n(U_n)\)\\
&=&2-\limsup_{k\to \infty}\tilde{\mathcal{E}}_n(U_n)\geq 2.
\end{eqnarray*}
Therefore, this is a contradiction.
So, we have $\mu\in (0,M_0-\de_0]$.
We now use Proposition \ref{prop:cc}.
Thus, for every $l\in\N$, we have $v_{1,k},\cdots,v_{l,k}$ and $w_{l,k}$ which satisfy the properties of Proposition \ref{prop:cc}.
Since $||v_{j,k}||_{L^2}^2\to \mu_{j}\leq M_0-\de_0$, we have by the Gagliardo-Nirenberg inequality on $\R^2$ that
\begin{eqnarray*}
\int_{\R^2}|\nabla v_{j,k}|^2-\frac{1}{2}|v_{j,k}|^4\geq c\int_{\R^2}|\nabla v_{j,k}|^2,
\end{eqnarray*}
for some $c>0$.
Take $l_1$ sufficiently large such that $c||\nabla v_1||_{L^2}^2>||w_{l_1,k}||_{L^4(\T_k^2)}^4$ for sufficiently large $k$.
Then we have
\begin{eqnarray*}
0&\geq &\limsup_{k\to \infty}\tilde{\mathcal{E}}_n(U_n)\\
&\geq &\limsup_{k\to \infty}\(\sum_{j=1}^{l_1}\tilde{\mathcal{E}}_n(v_{j,k})+\tilde{\mathcal{E}}_n(w_{l_1,k})\)\\
&\geq & \limsup_{k\to \infty}\(c \sum_{j=1}^{l_1}||\nabla v_{j,k}||_{L^2}^2-||w_{l_1,k}||_{L^4(\T_k^2)}^4\)\\
&\geq &  c||\nabla v_0||_{L^2}^2-||w_{l_1,k}||_{L^4(\T_k^2)}^4>0.
\end{eqnarray*}
This is a contradiction.
Therefore, we have the conclusion of the Theorem.
\end{proof}

Using Theorem \ref{thm:concentration}, we can show Corollary \ref{cor:nonexistence} by following the argument by Glangetas and Merle \cite{GM94b}.
Before proving Corollary \ref{cor:nonexistence}, we introduce a sharp Gagliardo-Nirenberg inequality on $\T^2$.

\begin{Thm}[\cite{CM08}]\label{thm:gn}
Let $u\in H^1(\T^2)$.
Then there exists $B>0$ such that 
\begin{eqnarray*}
\int_{\T^2}|u|^4\leq \(2M_0^{-1}\int_{\T^2}|\nabla u|^2\,dx+B\int_{\T^2}|u|^2\,dx\)\(\int_{\T^2}|u|^2\,dx\).
\end{eqnarray*}
\end{Thm}

Using Theorem \ref{thm:gn}, we have the following lemma.

\begin{Lem}\label{lem:energy}
Let $\mathcal{E}$ be the energy defined in $(\ref{eq:energy})$ and let $B$ given in Theorem \ref{thm:gn}.
Then, we have the following.
\begin{enumerate}
\item[$(\mathrm{i})$]
If $||u_0||_{L^2(\T^2)}^2=M_0-\ep_0$, $\ep_0>0$, then 
\begin{eqnarray*}
\mathcal{E}(u,n,v)+M_0B||u||_{L^2}^2\sim_{\ep_0} ||u||_{H^1}^2+||n||_{L^2}^2+||v||_{L^2}^2.
\end{eqnarray*}
\item[$(\mathrm{ii})$]
If $||u_0||_{L^2(\T^2)}^2=M_0$, then
\begin{eqnarray*}
\int_{\T^2}\(|u|^2+n\)^2\,dx+||v||_{L^2}^2\lesssim \mathcal{E}(u,n,v)+\frac{M_0^2B}{2}\lesssim ||u||_{H^1}^2+||n||_{L^2}^2+||v||_{L^2}^2.
\end{eqnarray*}
\end{enumerate}

\end{Lem}

\begin{proof}
(i): First, by Theorem \ref{thm:gn}, we have
\begin{eqnarray*}
\left|\int_{\T^2}n|u|^2\,dx\right|&\leq &||n||_{L^2}||u||_{L^4}^2\\
&\leq &\frac{||n||_{L^2}^2}{2(1+2\de)}+\(\frac{1}{2}+\de\)\(M_0-\ep_0\)\(\frac{2||\nabla u||_{L^2}^2}{M_0}+B||u||_{L^2}^2\).
\end{eqnarray*}
By taking $\de>0$ sufficiently small, we have
\begin{eqnarray*}
\left|\int_{\T^2}n|u|^2\,dx\right| \leq \(\frac{1}{2}-\frac{\ep_0}{2}\)||n||_{L^2}^2+\(1-\frac{\ep_0}{2}\)\int_{\T^2}|\nabla u|^2\,dx+\frac{M_0B}{2}||u||_{L^2}^2. 
\end{eqnarray*}
Therefore, we have
\begin{eqnarray*}
\mathcal{E}(u,n,v)+M_0B||u||_{L^2}^2\geq  \frac{\ep_0}{2}\Lebn{\nabla u}{2}^2+\frac{M_0B}{2}||u||_{L^2}^2+\frac{\ep_0}{2}\Lebn{n}{2}^2+\frac{1}{2}\Lebn{v}{2}^2.
\end{eqnarray*}
This implies that $\mathcal{E}(u,n,v)+M_0B||u||_{L^2}^2$ is equivalent to the square of the norm of $H^1\times L^2 \times L^2$.

(ii): By Theorem \ref{thm:gn}, we have
\begin{eqnarray*}
||\nabla u||_{L^2}^2-\frac{1}{2}||u||_{L^4}^4+\frac{M_0^2B}{2}\geq 0.
\end{eqnarray*}
Therefore, by the following identity, we have the conclusion.
\begin{equation*}
\mathcal{E}(u,n,v)=||\nabla u||_{L^2}^2-\frac{1}{2}||u||_{L^4}^4+\frac{1}{2}\int_{\T^2}\(|u|^2+n\)^2\,dx+\frac{1}{2}||v||_{L^2}^2.
\qedhere
\end{equation*}
\end{proof}

\begin{proof}[Proof of Corollary \ref{cor:nonexistence}]
We first consider the case $||u_0||_{L^2(\T^2)}< ||Q||_{L^2(\R^2)}$,
Since the energy $\mathcal{E}$ satisfies
\begin{eqnarray*}
\frac{d}{dt}\mathcal{E}(u,n,v)=\int_{\T^2}\hat{n}_1(0)(n+|u|^2)\lesssim 1+ \mathcal{E}(u,n,v),
\end{eqnarray*}
where $v=-\nabla^{-1}\(\d_t n-\hat{n}_1(0)\)$.
By Gronwall's inequality, we have
\begin{eqnarray*}
\mathcal{E}(u(t),n(t),v(t))\lesssim_T 1,\ t\in[0,T).
\end{eqnarray*}
Therefore, by Lemma \ref{lem:energy}, we have the conclusion.

For the case $||u_0||_{L^2(\T^2)}= ||Q||_{L^2(\R^2)}$, following Glangetas and Merle \cite{GM94b}, we argue by contradiction.
So, we assume there exists $T>0$ such that
\begin{eqnarray*}
||u(t)||_{H^1}+||n(t)||_{L^2}+||v(t)||_{L^2}\to\infty,\ t\to T.
\end{eqnarray*}
Now, as the previous case, we have
\begin{eqnarray*}
\mathcal{E}(u(t),n(t),v(t))\lesssim_T 1,\ t\in[0,T).
\end{eqnarray*}
Therefore, we have
\begin{eqnarray*}
||u(t)||_{L^2}^2+\int_{\T^2}(n(t)+|u(t)|^2)^2+||v(t)||_{L^2}^2\lesssim_T 1,\ t\in[0,T).
\end{eqnarray*}
Next, for $t\in[0,T)$,
\begin{eqnarray*}
||n||_{H^{-1}}&\lesssim &1+\int_0^t||n_t(s)||_{H^{-1}}\,ds\\
&\lesssim & 1+\int_0^t||\nabla v(s)||_{H^{-1}}\,ds\\
&\lesssim & 1+\int_0^t||v(s)||_{L^2}\,ds\lesssim 1.
\end{eqnarray*}
Therefore, we have
\begin{eqnarray*}
|| |u(t)|^2||_{H^{-1}}&\leq & ||n(t)||_{H^{-1}}+||n(t)+|u(t)|^2||_{H^{-1}}\\
&\lesssim & 1+ ||n(t)+|u(t)|^2||_{L^2}\lesssim 1.
\end{eqnarray*}
Finally, by Theorem \ref{thm:concentration}, we have $|u(t_n,x-x_n)|^2\rightharpoonup M_0\de_{x=0}$ as $t_n\to T$ in the distribution sense, where $\de_{x=0}$ is a delta function.
On the other hand, $|u(t_n,x-x_n)|^2$ is bounded in $H^{-1}$. So, by taking a subsequence of $|u(t_n,x-x_n)|^2$, it has a weak limit.
This implies $M_0\de_{x=0}\in H^{-1}$.
However since $\de_{x=0}\notin H^{-1}$, this is a contradiction.
\end{proof}

In the case $||u_0||_{L^2}<M_0$ and $n_1\in \hat{H}^{-1}$, we have the conservation of energy.
This implies that $||(u,n,n_t)||_{H^1\times L^2\times \hat{H}^{-1}}$ is bounded globally in time.
However, in the case $||u_0||_{L^2}<M_0$ and $n_1\in H^{-1}\setminus\hat{H}^{-1}$, there exists a global grow-up solution.

\begin{Prop}
For arbitrary $M$, there exists a time global solution $(u,n)$ of $($\ref{eq:Z}$)$ with $||u||_{L^2}^2=M$ such that $||(u(t),n(t),n_t(t))||_{{H^1}\times L^2\times H^{-1}}\to  \infty$ as $t\to \infty$.
\end{Prop}

\begin{proof}
One can easily check that $(u(t),n(t))=(ae^{-\frac{iat^2}{2}}, bt)$ is the solution of (\ref{eq:Z}) for any $a,b\in\R$.
Since $||n(t)||_{L^2}\sim t$, we have the conclusion of the Proposition.
\end{proof}

\appendix
\section{Proof of Proposition \ref{prop:cc}}
Proposition \ref{prop:cc} is a small modification of Proposition 1.7.6 of \cite{CazSemi}.
However, for the readers convenience, we give a sketch of the proof here.

In this section, we assume that $u_n$, $\la_n$ always satisfy the assumption of Proposition \ref{prop:cc}.
Set
\begin{eqnarray*}
\rho_n(u_n,R):=\sup_{y\in\T_n^2}\int_{|x-y|<R}|u(x)|^2\,dx.
\end{eqnarray*}

\begin{Lem}
There exists a subsequence of $u_n$ (which we also denote as $u_n$) such that there exists $R_n\to\infty$ such that $R_n<\la_n/4$ and
\begin{eqnarray*}
\mu=\lim_{n\to\infty}\rho_n(u_n,R_n)=\lim_{n\to\infty}\rho_n(u_n,R_n/2).
\end{eqnarray*}
\end{Lem}

\begin{proof}
We only show that we can take $R_n<\la_n/4$.
The rest of the proof is same as the proof of Lemma 1.7.5 of \cite{CazSemi}.

First, set
\begin{eqnarray*}
\tilde{\rho}_n(u_n,R):=\sup_{y\in\T_n^2}\int_{|x-y|_{\infty}<R}|u(x)|^2\,dx,
\end{eqnarray*}
where $|x|_{\infty}=\mathrm{max}\{|x_1|,|x_2|\}$ for $x=(x_1,x_2)$.
Then, from the assumption of $u_n$, we have $\tilde{\rho}_n(u_n,\la_n)\to M$.
Therefore, following the proof of Lemma 1.7.5 of \cite{CazSemi}, we can show that there exists $\tilde{R}_n$ such that
\begin{eqnarray*}
\mu=\lim_{n\to\infty}\tilde{\rho}_n(u_n,\tilde{R}_n)=\lim_{n\to\infty}\tilde{\rho}_n(u_n,\tilde{R}_n/8).
\end{eqnarray*}
So, since
$
\{|x|_{\infty}<\tilde{R}_n/8\}\subset \{|x|<\tilde{R}_n/4\}\subset \{|x|_{\infty}<\tilde{R}_n/4\},
$
we have
$\mu=\lim_{n\to\infty} \rho_n(u_n,\tilde{R}_n/4).
$
Therefore, taking $R_n:=\tilde{R}_n/4$, we have the conclusion.
\end{proof}

\begin{Lem}\label{lem:177}
There exists a constant $K$ which is independent of $n$ such that
\begin{eqnarray*}
\int_{\T_n^2}|u_n|^4\leq K\rho_n(u_n,1)||u||_{H^1(\T_n^2)}^2.
\end{eqnarray*}
\end{Lem}

\begin{proof}
See the proof of Lemma 1.7.7 of \cite{CazSemi}.
\end{proof}

\begin{proof}[Proof of Proposition \ref{prop:cc}]
(i) follows from Lemma \ref{lem:177}.
We show (ii).
Let $\theta\in C^{\infty}(\T^2)$ such that $0\leq \theta\leq 1$ and 
$\theta(x)= 1$ for $0\leq |x|_{\infty}\leq 1/2$ and $\theta(x)=0$ for $|x|_{\infty}\geq 3/4$.
Further, define $\theta_n, \varphi_n\in C^{\infty}(\T_n^2)$ as
\begin{eqnarray*}
\theta_n(x)=\theta\(\frac{|x-y_n(R_n/2)|}{R_n}\),\ \varphi_n(x)=1-\theta_n(x/2).
\end{eqnarray*}
Then, following the proof of Proposition 1.7.6 of \cite{CazSemi}, we see that $v_{1,n}=\theta_nu_n$, $w_{1,n}=\varphi_nu_n$ satisfies (ii) (b-f) for the case $l=1$.
By using the same argument to $w_{1,n}$ we have (ii) (b-f) for the case $l=2$.
So, iteratively, we have (ii) (b-f) for all $l\geq 1$.

We only have to prove (ii) (a).
Now, since $\mathrm{v_{j,k}}\subset \{x\in T_n^2 | |x-y_n(R_n/2)|<R_n\}\subset \{x\in T_n^2 | |x-y_n(R_n/2)|<\la_n/4\}$, so we can consider $v_{j,n}$ as a function on $\R^2$.
Therefore, by a direct application of Proposition 1.7.6 of \cite{CazSemi}, it suffices to show that
\begin{eqnarray*}
\tilde{\mu}_j:=\lim_{R\to\infty}\liminf_{n\to \infty}\sup_{y\in\T_n^2}\int_{|x-y|<R}|v_{j,k}(x)|^2=\mu_j.
\end{eqnarray*}
We only show this for the case $j=1$.
Suppose $\tilde{\mu}_1<\mu$.
By taking $l\geq 1$ sufficiently large, we have $\mu>\mu_l$.
If there exists $2\leq j\leq l$ such that $\tilde{\mu}_j=\mu$, then just rename $v_{j,k}$ as $v_{1,k}$.
Therefore, we have $\mu>\tilde{\mu}_j$ for $j=1,\cdots,l$.
Let $\ep:=(\mu-\max_{j=1,\cdots,l}\{\tilde{\mu}_j,\mu_l\})/2$.
Then, for sufficiently large $R$, we have
\begin{eqnarray*}
\mu-\ep&<&\liminf_{n\to\infty}\sup_{y\in\T_n^2}\int_{|x-y|<R}|u_n|^2\\
&=&\liminf_{n\to\infty}\sup_{y\in\T_n^2}\int_{|x-y|<R}\sum_{j=1}^l|v_{j,n}|^2+|w_{l,n}|^2\\
&=&\max_{j=1,\cdots,l}\{\tilde{\mu}_j,\mu_l\},
\end{eqnarray*}
where the second equality follows from (ii) (d) and the third equality follows from the fact that
\[ \mathrm{dist}(\mathrm{supp}v_{j,n},\mathrm{supp}v_{j',n})>3R\]
for $j\neq j'$ and
\[ \mathrm{dist}(\mathrm{supp}v_{j,n},\mathrm{supp}w_{l,n})>3R\]
for sufficiently large $n$.
This inequality implies $2\ep<\ep$.
This is a contradiction and we have the conclusion.
\end{proof}


\section*{Acknowledgments}
\noindent N. Kishimoto  acknowledges JSPS the Grant-in-Aid for Young Scientists (B) 24740086.
M. Maeda acknowledges JSPS the Grant-in-Aid for Young Scientists (B) 24740081.  %

\def\cprime{$'$}
\providecommand{\bysame}{\leavevmode\hbox to3em{\hrulefill}\thinspace}
\providecommand{\MR}{\relax\ifhmode\unskip\space\fi MR }
\providecommand{\MRhref}[2]{%
  \href{http://www.ams.org/mathscinet-getitem?mr=#1}{#2}
}
\providecommand{\href}[2]{#2}

\bigskip
{\small \textit{E-mail addresses}: \texttt{n-kishi@math.kyoto-u.ac.jp},\ \ \texttt{m-maeda@math.tohoku.ac.jp}}

\end{document}